\documentclass{article}
\usepackage[a4paper,outer=3.5cm,inner=3.5cm,total={14cm,21.6cm}]{geometry} 
\usepackage[english]{babel}
\usepackage[utf8]{inputenc}
\usepackage{amsmath}
\usepackage{amssymb}
\usepackage{amsthm}
\usepackage{enumerate}
\usepackage[shortlabels]{enumitem}
\usepackage{tikz-cd}
\usepackage{stmaryrd}

\newtheorem{thm}{Theorem}[section]
\newtheorem{Theorem}[thm]{Theorem}
\newtheorem{Lemma}[thm]{Lemma}
\newtheorem{Proposition}[thm]{Proposition}
\newtheorem{Corollary}[thm]{Corollary}
\newtheorem{claim}[thm]{Claim}

\theoremstyle{definition}
\newtheorem{Definition}[thm]{Definition}
\newtheorem{Question}[thm]{Question}
\newtheorem{Example}[thm]{Example}
\newtheorem{Remark}[thm]{Remark}

\newcommand{\B}{\mathbb{B}}
\newcommand{\C}{\mathbb{C}}

\newcommand{\G}{\mathbb{G}}
\newcommand{\N}{\mathbb{N}}
\renewcommand{\O}{\mathcal{O}}
\renewcommand{\P}{\mathbb{P}}
\newcommand{\Q}{\mathbb{Q}}

\newcommand{\Z}{\mathbb{Z}}
\newcommand{\A}{\mathbb{A}}
\newcommand{\m}{\mathfrak{m}}

\newcommand{\coker}{\operatorname{coker}}
\newcommand{\im}{\operatorname{im}}

\newcommand{\Hom}{\operatorname{Hom}}
\newcommand{\Map}{\operatorname{Map}}

\newcommand{\End}{\operatorname{End}}

\newcommand{\Gal}{\operatorname{Gal}}
\newcommand{\GL}{\operatorname{GL}}
\newcommand{\Perf}{\operatorname{Perf}}

\newcommand{\Spa}{\operatorname{Spa}}
\newcommand{\Spf}{\operatorname{Spf}}
\newcommand{\Spec}{\operatorname{Spec}}
\newcommand{\Spd}{\operatorname{Spd}}

\newcommand{\Rlim}{\mathrm{R}\!\lim}

\newcommand{\cH}{{\ifmmode \check{H}\else{\v{C}ech}\fi}}
\newcommand{\tf}{[\tfrac{1}{p}]}
\renewcommand{\tt}{\mathrm{tt}}

\newcommand{\dR}{\mathrm{dR}}
\newcommand{\VB}{\mathrm{VB}}
\newcommand{\BdR}{{\mathbb B_{\dR}}}
\newcommand{\BdRp}{{\mathbb B^+_{\dR}}}

\newcommand{\id}{{\operatorname{id}}}
\newcommand{\cts}{{\operatorname{cts}}}
\newcommand{\lc}{{\operatorname{lc}}}
\newcommand{\an}{\mathrm{an}}
\newcommand{\hotimes}{\hat{\otimes}}

\newcommand{\et}{{\mathrm{\acute{e}t}}}
\newcommand{\proet}{\mathrm{pro\acute{e}t}}
\newcommand{\profet}{\mathrm{prof\acute{e}t}}
\newcommand{\qproet}{\mathrm{qpro\acute{e}t}}

\newcommand{\HT}{\operatorname{HT}}
\newcommand{\aeq}{\stackrel{a}{=}}

\newcommand{\wh}{\widehat}

\newcommand{\bOx}{\overline{\O}^{\times}}

\newcommand{\Pic}{\operatorname{Pic}}

\newcommand{\wt}{\widetilde}
\renewcommand{\diamond}{\diamondsuit}

\renewcommand{\lim}{\varprojlim}

\newcommand*\isomarrow{%
	\xrightarrow{\raisebox{-0.35em}{\smash{\ensuremath{\sim}}}}
}  

\tikzset{
	labelrotate/.style={anchor=south, rotate=90, inner sep=.5mm}} 
\tikzset{
	labelrotatep/.style={anchor=north, rotate=90, inner sep=.75mm}}

\begin{document}
\title{Line bundles on rigid spaces in the $v$-topology}
\author{Ben Heuer}
\date{}
\maketitle
\begin{abstract}
	For a smooth rigid space $X$ over a perfectoid field extension $K$ of $\Q_p$, we investigate how the $v$-Picard group of the associated diamond $X^\diamond$ differs from the analytic Picard group of $X$. To this end, we construct a left-exact ``Hodge--Tate logarithm'' sequence \[0\to \Pic_{\an}(X)\to \Pic_v(X^\diamond)\to H^0(X,\Omega_X^1)\{-1\}.\]
	We deduce some analyticity criteria which have applications to $p$-adic modular forms.
	
	For algebraically closed $K$, we show that the sequence is also right-exact if $X$ is proper or one-dimensional. In contrast, we show that for the affine space  $\A^n$, the image of the Hodge--Tate logarithm consists precisely of the closed differentials.
	
	It follows that up to a splitting, $v$-line bundles may be interpreted as Higgs bundles. For proper $X$, we use this to construct the $p$-adic Simpson correspondence of rank one.
\end{abstract}

\section{Introduction}
Let $p$ be a prime and let $K$ be a perfectoid field extension of $\Q_p$, e.g.\ we could take $K=\C_p$. Let $X$ be a smooth rigid space over $K$, considered as an adic space. Then there is a hierarchy of topologies on $X$:
\[X_{\an}\subseteq X_{\et}\subseteq X_{\proet},\]
where $X_{\proet}$ is the pro-\'etale site defined by Scholze in \cite[Definition~3.9]{Scholze_p-adicHodgeForRigid}.

It is a natural question whether the notions of vector bundles agree in these various topologies: To make this precise, let us denote by $\VB_{\tau}(X)$ the category of finite locally free modules over the structure sheaf where $\tau$ is any of the above topologies. Here  for the pro-\'etale topology, we use the completed structure sheaf \cite[Definition 4.1]{Scholze_p-adicHodgeForRigid}.

By a rigid version of \'etale descent (see \cite[Proposition~8.2.3]{FvdP}), the natural functor
\[ \mathrm{VB}_{\an}(X)\isomarrow \mathrm{VB}_{\et}(X)\]
is an equivalence of categories. 
One may ask if a similar result holds for
\[ \mathrm{VB}_{\et}(X)\to  \mathrm{VB}_{\proet}(X).\]

\begin{Question}\label{q:howfarequiv}
How far is this functor from being an equivalence of categories?
\end{Question}
It is easy to see that an equivalence would be too much to ask for: As has been observed in the literature \cite[before \S1.2]{CHJ}, descent of analytic vector bundles along pro-\'etale covers is in general not effective, giving rise to ``new'' vector bundles in the pro-\'etale topology. It is known that pro-\'etale vector bundles arise naturally in the context of $p$-adic modular forms, as well as in the $p$-adic Simpson correspondence \cite[\S2]{LiuZhu_RiemannHilbert}\cite[\S3]{wuerthen_vb_on_rigid_var}\cite[\S7]{MannWerner_LocSys_p-adVB}. However,  a systematic description of these  additional vector bundles has not yet been given.

A similar question arises in a slightly different setting:
Recently, Scholze has constructed the category of diamonds  \cite[\S11]{etale-cohomology-of-diamonds}, into which both semi-normal rigid spaces over $K$ as well as perfectoid spaces embed fully faithfully by way of  a diamondification functor \cite[Proposition 10.2.3]{ScholzeBerkeleyLectureNotes}
\[ X\mapsto X^{\diamondsuit}.\]
While the \'etale cohomology of diamonds has been  studied in great detail \cite{etale-cohomology-of-diamonds}, vector bundles on diamonds are less well-understood, even for diamonds associated to rigid spaces. 

The category of (locally spatial) diamonds can be equipped with three well-behaved topologies: The \'etale, quasi-pro-\'etale and $v$-topology. If $X$ is a smooth rigid space, then for the \'etale topology, there is an equivalence of sites 
$X_{\et}=X^{\diamondsuit}_{\et}$ \cite[Theorem 10.4.2]{ScholzeBerkeleyLectureNotes} that identifies the structure sheaves. We can thus extend the above hierarchy of topologies to
\[ X^{\phantom\diamondsuit}_{\an}\subseteq X^{\phantom\diamondsuit}_{\et}= X^{\diamondsuit}_{\et}\subseteq X^{\phantom\diamondsuit}_{\proet}\subseteq X^{\diamondsuit}_{\qproet}\subseteq X^{\diamondsuit}_v.\]

If $X$ is an affinoid perfectoid space over $K$, then the notions of vector bundles agree for all of these various topologies by a result of Kedlaya--Liu \cite[Theorem~3.5.8]{KedlayaLiu-II}. Since the last three of the above sites are locally affinoid perfectoid, it follows that also for smooth rigid $X$, we have
\[\VB_{\proet}(X)= \VB_{\qproet}(X^\diamondsuit)= \VB_{v}(X^\diamondsuit).\]
We may therefore in the following often use these three topologies interchangeably. 
We conclude that also in this more refined setting, the various notions of vector bundles agree, except potentially for the step from the \'etale to the pro-\'etale site.

In the case of line bundles, we can make Question~\ref{q:howfarequiv} more precise by passing to the Picard group of isomorphism classes, and consider the natural homomorphism
\begin{equation}\label{eq:intro-Pic_an->Pic_v}
	\Pic_{\an}(X)=\Pic_{\et}(X)\hookrightarrow \Pic_{\proet}(X) =\Pic_{v}(X^\diamondsuit).
\end{equation}

\begin{Question}\label{q:howfariso}
	Can we explicitly describe the cokernel of this map?
\end{Question}

\subsection{The Hodge--Tate logarithm}
The first goal of this article is to answer Question~\ref{q:howfariso} for smooth rigid spaces over $K$. For these we show that the Picard group of $v$-line bundles admits a $p$-adic Hodge-theoretical description in terms of differentials that we regard as a ``multiplicative Hodge--Tate sequence'':
\begin{Theorem}\label{t:Picv-ses}
	Let $K$ be a perfectoid field over $\Q_p$. Let $X$ be a smooth rigid space over $K$.
	\begin{enumerate}
		\item The $p$-adic logarithm defines a natural left-exact sequence, functorial in $X$,
		\begin{equation}\label{eq:MT-ses}
		0\to \Pic_{\an}(X)\to \Pic_v(X)\xrightarrow{ \mathrm{HT}\log} H^0(X,\Omega^1_X)\{-1\}.
		\end{equation}
		\item\label{enum:res}  If $K$ is algebraically closed, the sequence is right-exact in either of the following cases:
		\begin{enumerate}
			\item\label{enum:MT1-projective} $X$ is proper, or
			\item\label{enum:MT1-curves} $X$ is of pure dimension 1 and paracompact.
		\end{enumerate}
		\item If $X$ is affinoid, the sequence becomes right-exact after inverting $p$.
	\end{enumerate}
\end{Theorem}
\begin{Remark}
The $\{-1\}$ in Theorem~\ref{t:Picv-ses} is a Breuil--Kisin--Fargues twist (see Definition~\ref{d:BK-twist}) that can be identified with a Tate twist $(-1)$ if $K$ contains all $p$-power unit roots. In any case, one can always choose a distinguished element for $K$ to fix an isomorphism $\Omega^1_X\{-1\}\cong \Omega^1_X$.

We note that if $K$ is not perfectoid, already $\Pic_v(\Spa(K))$ is in general very large.
\end{Remark}

As an application of Theorem~\ref{t:Picv-ses}.1, we get some useful criteria for telling whether a $v$-line bundle is analytic:
For example, these can be applied to see that the sheaf of overconvergent modular forms of Chojecki--Hansen--Johansson \cite{CHJ} is analytic (see Example~\ref{exmp:CHJ}).
\begin{Corollary}\label{c:analytic-on-Zar-dense-implies-analytic-intro}
	 Let $L$ be a $v$-line bundle on $X$. Let $V\subseteq X$ be any Zariski-dense analytic open subspace. Then $L$ is analytic (i.e.\ a line bundle on $X_{\an}$) if and only if $L|_V$ is analytic.
\end{Corollary}
	\begin{Corollary}\label{c:non-trivial-section-implies-analytic-intro}
		Assume that $X$ is connected
	and let $L$ be a $v$-line bundle on $X$. If we have $H^0(X,L)\neq 0$, then $L$ is analytic.
\end{Corollary}

For the rest of this introduction, let us assume that $K$ is algebraically closed.
In order to shed some light on how the additional $v$-topological line bundles arise, we consider the case that $X$ is an abelian variety $A$ over $K$. Then there is a pro-finite-\'etale perfectoid cover \[\wt A:=\textstyle\varprojlim_{[p]} A\to A\]
with Galois group $T_pA$. Its Cartan--Leray spectral sequence gives rise to an exact sequence
\[ 0\to \Hom_{\cts}(T_pA,K^\times)\to \Pic_v(A)\to \Pic_{\an}(\wt A)^{T_pA},\]
where the first term can be interpreted as pro-\'etale descent data on the trivial line bundle on $\wt A$. On the other hand, we have a logarithmic Hodge--Tate map
\[\Hom_{\cts}(T_pA,K^\times)\xrightarrow{\log}\Hom_{\cts}(T_pA,K) \xrightarrow{\mathrm{HT}} H^0(A,\wt\Omega^1_A),\]
which is surjective. Combining these observations proves Theorem~\ref{t:Picv-ses}.\ref{enum:MT1-projective} for $A$.

The strategy for the general proper case is similar to this: We introduce a diamantine universal pro-finite-\'etale cover
\[\wt X\to X\]
obtained by taking the limit over all connected finite \'etale covers in the category of diamonds. One can then show using Scholze's $p$-adic Hodge theory that the descent data for the trivial line bundle on $\wt X$ generate the group $H^0(X,\wt \Omega^1_X)$ under the Hodge--Tate logarithm.

\subsection{The $p$-adic Simpson correspondence for line bundles}
	The proper case of
	Theorem~\ref{t:Picv-ses} is very closely related to the still mostly conjectural $p$-adic Simpson correspondence \cite{Faltings_SimpsonI}\cite{DeningerWerner_vb_p-adic_curves}: Namely, the theorem shows that we may interpret $v$-topological line bundles on $X$ as Higgs bundles of rank $1$  on $X$, up to a choice of splitting.
	
	On the other hand, as has been observed before by Liu--Zhu \cite[Remark~2.6]{LiuZhu_RiemannHilbert}, pro-\'etale vector bundles are closely related to generalised representations in the sense of Faltings. As our second application, we discuss in \S4 how Theorem~\ref{t:Picv-ses}.\ref{enum:MT1-projective} can thus be used to construct the $p$-adic Simpson correspondence in rank $1$, again based on the diamantine universal cover:
	\begin{Theorem}	
		Let $X$ be a connected smooth proper rigid space over $K$. Fix $x\in X(K)$.
		Then there is an equivalence of tensor categories
	\[\Bigg\{\begin{array}{@{}c@{}l}\text{$1$-dim.\ continuous $K$-linear}\\\text{representations of $\pi_1(X,x)$} \end{array}\Bigg\} \isomarrow \Bigg\{\begin{array}{@{}c@{}l}\text{pro-finite-\'etale analytic}\\\text{Higgs bundles on $X$ of rank }1\end{array}\Bigg\},\]
	depending on choices of a Hodge--Tate splitting and of an exponential function.
	\end{Theorem}
	The construction is entirely global, and avoids any localisation steps to affinoid opens.
	Apart from Theorem~\ref{t:Picv-ses}, our main innovation for the proof is the introduction of the diamantine universal cover $\wt X\to X$, which is a good replacement for the topological universal cover in complex geometry and its role in the complex Simpson correspondence \cite{SimpsonCorrespondence}.
	
	This also gives some idea what to expect for generalisations of Theorem~\ref{t:Picv-ses} to higher ranks. In particular, we believe that the perspective provided by the universal cover $\wt X$ may help find the correct subcategory of Higgs bundles on $X$ for the general $p$-adic Simpson correspondence, which so far has not yet been identified in general.
\subsection{Affine space and affinoid spaces}
The $p$-adic Hodge theoretic phenomenon explaining surjectivity of \eqref{eq:MT-ses} only occurs in the proper case, and indeed for non-quasi-compact spaces it is rarely true that line bundles are trivialised by pro-finite-\'etale covers. In order to investigate what we can expect beyond the proper case, we also determine the $v$-Picard group of the rigid affine space $\A^n$ over $K$:
\begin{Theorem}\label{t:intro-Picv-A^n}
For any $n\in \N$, the Hodge--Tate logarithm defines an isomorphism
\[ \Pic_v(\A^n)=H^0(\A^n,\Omega^1(-1))^{d=0}.\]
\end{Theorem}
This ties in with recent results of Colmez--Nizio\l\ \cite{ColmezNiziol_CohomologyAffine} and Le Bras \cite{LeBras-Espaces} describing the pro-\'etale cohomology of $\A^n$. We offer two proofs of Theorem~\ref{t:intro-Picv-A^n}, one relying on a comparison to Le Bras' work, the other an independent one which is slightly more general.

We emphasize that in contrast to Theorem~\ref{t:Picv-ses}.2, the Picard group only sees the closed differentials, rather than all of $H^0(\A^n,\Omega^1(-1))$. In particular, this implies that right-exactness in Theorem~\ref{t:Picv-ses} fails for affinoid $X$ already in the case of the closed disc of radius $\geq 2$.

Instead, for affinoid $X$, the image of the Hodge--Tate logarithm is a slightly mysterious subgroup of $H^0(X,\Omega^1(-1))$ that contains all closed differentials and generates the whole space upon inverting $p$, but is typically not all of $H^0(X,\Omega^1(-1))$ unless $X$ is one-dimensional.
\subsection*{Acknowledgements}
We thank Johannes Ansch\"utz, David Hansen, Arthur-C\'esar Le Bras, Lucas Mann, Peter Scholze and Annette Werner for very useful conversations. 

This work was funded by the Deutsche Forschungsgemeinschaft (DFG, German Research Foundation) under Germany's Excellence Strategy -- EXC-2047/1 -- 390685813. The author was supported by the DFG via the Leibniz-Preis of Peter Scholze.

\subsection*{Notation}
Throughout, let $K$ be a perfectoid field extension of $\Q_p$. Let $\O_K$ be the ring of integers, $\m$ its maximal ideal, $k$ the residue field.  Let $C$ be the completion of an algebraic closure of $K$.

 We use almost mathematics with respect to $(\O_K,\mathfrak m)$ and write $\aeq$ if a natural map becomes an isomorphism after passing to the almost category.

By a rigid space over $K$ we shall by definition mean an adic space in the sense of Huber \cite{huber2013etale} that is locally of topologically finite type over $\Spa(K,\O_K)$. 

Let $\Perf_K$ be the category of perfectoid spaces over $K$. 
Throughout, we shall consider diamonds over $\Spa(K,\O_K)$ in the sense of \cite{etale-cohomology-of-diamonds}, which in this relative setting we may consider as $v$-sheaves on $\Perf_K$.
We recall the diamondification functor \cite[\S10.1]{ScholzeBerkeleyLectureNotes}
\[ \{\text{analytic adic spaces over }K\}\to \{\text{diamonds over }\Spd(K)\},\quad X\mapsto X^\diamondsuit\]
which is fully faithful on semi-normal rigid spaces. 
For any adic space $X$, we will write $X^\diamondsuit$ for the associated diamond when we would like to emphasize the category we work in. However, we often drop this from notation and identify semi-normal rigid space and perfectoid space with their associated diamonds when this is clear from the context.

For a smooth rigid space $X$, we denote by $X_{\proet}$ the pro-\'etale site in the sense of \cite[Definition~3.9]{Scholze_p-adicHodgeForRigid}, which is now sometimes referred to as the ``flattened pro-\'etale site''.

Let us fix notation for some rigid groups we will use: $\G_a$ denotes the rigid analytic affine line $\A^1$ with its additive structure, $\G_a^+$ denotes the subgroup defined by the closed ball of radius 1 around the origin. $\G_m$ denotes the rigid analytic affine line  punctured at the origin with its multiplicative group structure.
We denote by $\O,\O^+,\O^\times$ the sheaves that these groups represent on the \'etale, pro-\'etale or $v$-site. We will indicate the topology by an index, e.g.\ $\O_\tau$ for $\tau=\et,\qproet,v$,... unless this is clear from the context.

\section{Vector bundles on diamonds}
In this section we prove Theorem~\ref{t:Picv-ses}.1 using the Leray spectral sequence of $\nu:X_v\to X_{\et}$ for the sheaf $\O^\times$. To avoid any ambiguity, we begin with a definition of $v$-vector bundles.
\subsection{Definition and basic properties}
For $n\in \N$, let $\GL_n^\diamondsuit$ be the diamond associated to $\GL_n$ considered as a rigid space over $K$.
\begin{Definition}
	Let $Y$ be a diamond over $\Spd(K)$. A $v$-vector bundle of rank $n\in \N$ on $Y$ is a $\GL_n^\diamondsuit$-torsor for the $v$-topology, i.e.\ a $v$-sheaf $V\to Y$ with a $\GL_n^\diamondsuit$-action $\GL_n^\diamondsuit\times V\to V$  over $Y$ for which there is a $v$-cover $Y'\to Y$ with a $\GL_n^\diamondsuit$-equivariant Cartesian diagram
	\[
	\begin{tikzcd}
	\GL_n^\diamondsuit\times Y' \arrow[d] \arrow[r,"\pi_2"] & Y' \arrow[d] \\
	V \arrow[r]                                & Y.           
	\end{tikzcd}\]
	As usual, one sees that this geometric definition is equivalent to the sheaf-theoretic one where a $v$-vector bundle is defined as a locally free $\O_v$-modules of rank $n$ on $Y_v$.
\end{Definition}
In the case of perfectoid spaces, the above $v$-topological notion of vector bundles is equivalent to the usual notion of vector bundles in the analytic topology:
\begin{Theorem}[Kedlaya--Liu {\cite[Theorem 3.5.8]{KedlayaLiu-II}}]\label{t:KL-line-bundles}
	Let $X$ be a perfectoid space over $K$. Then any $v$-vector bundle on $X$ is already trivial locally in the analytic topology on $X$.
\end{Theorem} 
As a consequence, vector bundles are in general trivial in the quasi-pro-\'etale topology:
\begin{Corollary}\label{c:refine-X-so-L-becomes-trivial}
	Let $Y$ be a diamond and let $V$ be a $v$-vector bundle on $Y$. Then there is a presentation $Y=X/R$ for some perfectoid space $X$ and some pro-\'etale equivalence relation $R\subseteq X\times X$ such that the pullback of $V$ to $X$ is trivial. In particular, any $v$-vector bundle on $Y$ is already trivial in the quasi-pro-\'etale topology.
\end{Corollary}
\begin{proof}
	Let $Y=X/R$ be any presentation, then by Theorem~\ref{t:KL-line-bundles}, there is an analytic cover $X'\to X$ such that the pullback of $V$ to $X$ becomes trivial over $X'$. Let $R'\subseteq X'\times X'$ be the base-change of $R\to X\times X$, then by \cite[Proposition~11.3.3-4]{etale-cohomology-of-diamonds}, this is again a pro-\'etale equivalence relation, and we have $X'/R'=X/R$.
\end{proof}
\begin{Corollary}\label{l:descent-of-line-bundles}
	Let $Y$ be a diamond. Then any $v$-vector bundle on $Y$ is a diamond.
\end{Corollary}
\begin{proof}
	Let $V$ be a $v$-vector bundle on $Y$. By Corollary~\ref{c:refine-X-so-L-becomes-trivial}, there is a quasi-pro-\'etale cover $Y'\to Y$ trivialising $V$. We thus have a quasi-pro-\'etale surjective morphism of $v$-sheaves
	$\GL_n^\diamond \times Y'\to V$
	from a diamond, so by \cite[Proposition 11.6]{etale-cohomology-of-diamonds} $V$ is itself a diamond.
\end{proof}
In particular, using that $X\to Y$ is a $v$-cover, one can describe $v$-vector bundles on $Y$ in terms of analytic vector bundles on $X$ equipped with descent data, as usual. More generally:
\begin{Definition}
	Let $q:X\to Y$ be a $v$-cover of diamonds. Write $\pi_1,\pi_2:X\times_Y X\rightrightarrows X$ for the projection maps. 
	 Let $ V$ be a $v$-vector bundle on $X$. Then a descent datum on $V$ with respect to $q$ is an isomorphism of $v$-vector bundles on $X\times_Y X$
	\[ \varphi:\pi_1^{\ast}V\isomarrow \pi_2^{\ast}V ,\]
	such that the cocycle condition holds.
	For a $v$-vector bundle $V_0$ on $Y=X/R$, the pullback along $q:X\to Y$ carries a canonical descent datum induced by $q\circ \pi_{1}=q\circ \pi_2$.
	A descent datum $\varphi$ is called effective if it is isomorphic to a descent datum of this form.
\end{Definition}
\begin{Lemma}\label{l:effective-descent}
	Let $q:X\to Y$ be a $v$-cover of diamonds. Then any descent datum on a $v$-vector bundle on $X$ is effective: The $v$-vector bundle on $Y$ attached to  $\varphi:\pi_1^{\ast}V\isomarrow \pi_2^{\ast}V$ is
	\[ V_0:=\ker( q_{\ast}V\xrightarrow{\pi_2^{\ast}-\varphi\circ\pi_1^{\ast}} q_{\ast}\pi_{2\ast}\pi_2^{\ast}V).\]
	  In particular, $v$-vector bundles of rank $n$ on $Y$ up to isomorphism are classified by the set
	\[ \Pic_v(Y):=H^1_v(Y,\GL_n^\diamondsuit).\]
\end{Lemma}

In the special case that the diamond $Y$ is the quotient of a perfectoid space $X$ by the action of a profinite group, the descent data defining vector bundles can be described as $1$-cocycles in continuous group cohomology, as we shall now discuss.
\subsection{The Cartan--Leray spectral sequence}
\begin{Definition}
	Let $f:X\to Y$ be a morphism of diamonds over $\Spd(K)$. Let $G$ be a locally profinite group. We say that $f$ is Galois with group $G$ if $f$ is a quasi-pro-\'etale $G$-torsor: Explicitly, this means that $f$ is a quasi-pro-\'etale cover and there is a $G$-action on $X$ that leaves $f$ invariant such that the action and projection maps induce an isomorphism
	\[G\times X\isomarrow X\times_YX.\]
\end{Definition}
Let $f:X\to Y$ be Galois with group $G$ and let $\mathcal F$ be a sheaf of topological abelian groups on $Y_v$.  Generalising from the case of finite $G$ known from \'etale cohomology, one might hope that there is in this situation a Cartan--Leray spectral sequence relating the continuous group cohomology of $H^j_v(X,\mathcal F)$ with $H^j_v(Y,\mathcal F)$. 
However, apart from special cases (e.g.\ if $\mathcal F$ is a sheaf of discrete abelian groups pulled back from $Y_{\et}$, see \cite[Remark~2.25]{CHJ}), it is not clear how to make this precise: Topological abelian groups do not form an abelian category, and it is in general not clear what topology $H^j_v(X,\mathcal F)$ should be endowed with. These issues can be fixed using the formalism of condensed abelian groups of Clausen--Scholze \cite{Scholze-Condensed}. 

For our purposes, however, the following ad hoc version in low degrees will be sufficient:
\begin{Proposition}\label{p:Cartan--Leray}
	Let $q:X\to Y$ be a morphism of diamonds over $K$ that is Galois for the action of a locally profinite group $G$ on $X$. Let $\tau=v$ or $\qproet$ and let $\mathcal F$ be a sheaf of not necessarily abelian topological groups on $Y_{\tau}$ with the property that 	for $i=1,2$ we have
	\begin{equation}\label{eq:CL-condition}
	\mathcal F(X\times G^i)=\Map_{\cts}(G^i,\mathcal F(X)).
	\end{equation}
	For example, for $\mathcal F=\O, \mathcal O^\times, \GL_n(\O)\dots$, this condition holds for any $i\geq 0$.
 Then:
 \begin{enumerate}
 \item\label{i:CL-left-ex} There is a left-exact sequence of pointed sets (of abelian groups if $\mathcal F$ is abelian):
	\[0\to H^1_{\cts}(G,\mathcal F(X))\to H^1_\tau(Y,\mathcal F)\to H^1_\tau(X,\mathcal F)^G.\]
	\item \label{i:CL-five-ex} Assume that $\mathcal F$ is abelian, and that moreover the natural specialisation map
	\[ H^1_\tau(X\times G,\mathcal F)\to \Map(G,H^1_\tau(X,\mathcal F))\]
	is injective. Then this extends to a ``Cartan--Leray $5$-term exact sequence''
	\[ 0\to H^1_{\cts}(G,\mathcal F(X))\to H^1_\tau(Y,\mathcal F)\to H^1_\tau(X,\mathcal F)^G\to H^2_{\cts}(G,\mathcal F(X))\to H^2_\tau(Y,\mathcal F). \]
	\item\label{i:CL-full} If moreover $H^j_\tau(X,\mathcal F)$ carries a topology for all $j\geq 1$ such that for all $i\geq 0$ we have
	\begin{equation}\label{eq:CL-higher}
	H^j_{\tau}(X\times G^i,\mathcal F)= \Map_{\cts}(G^i,H^j_{\tau}(X,\mathcal F)),
	\end{equation}
	then we obtain the full Cartan--Leray spectral sequence
	\[E^{ij}_2=H^i_{\cts}(G,H^j_{\tau}(X,\mathcal F))\Rightarrow H^{i+j}_{\tau}(Y,\mathcal F).\]
	\end{enumerate}
\end{Proposition}
The last part is implicit in \cite[\S5]{Scholze_p-adicHodgeForRigid} where it is used in the following form:
\begin{Corollary}\label{c:CL-acyclic}
If $\mathcal F$ satisfies $\eqref{eq:CL-condition}$  and is $\tau$-acyclic on $X\times G^i$ for all $i\geq 0$, then we have
\[H^i_{\cts}(G,\mathcal F(X))=H^{i}_{\tau}(Y,\mathcal F).\]
\end{Corollary}
\begin{proof}[Proof of Proposition~\ref{p:Cartan--Leray}]
	These all follow from the \cH-to-sheaf spectral sequence of the $\tau$-cover $X\to Y$. The associated \cH-complex is of the form
	\[ H^j_{\tau}(X,\mathcal F)\to  H^j_{\tau}(X\times G,\mathcal F)\to  H^j_{\tau}(X\times G\times G,\mathcal F) \dots   \]
	which by \eqref{eq:CL-condition} for $i=0,1,2$ and $j=0$ in part 1, respectively by \eqref{eq:CL-higher} in part 3, is equal to
	\begin{equation}\label{eq:cts-bar-complex}
	 = H^j_{\tau}(X,\mathcal F)\to \Map_{\cts}(G,H^j_{\tau}(X,\mathcal F))\to  \Map_{\cts}(G\times G,H^j_{\tau}(X,\mathcal F))\to \dots
	 \end{equation}
	By a standard computation,  this is precisely the complex of continuous cochains, which by definition computes $H^i_{\cts}(G,H^j_{\tau}(X,\mathcal F))$. This shows parts 1 and
	part 3.
	
	Part 2 holds as for the $5$-term exact sequence, we only need $\cH^i((X\to Y),\mathcal F)$ for $i=1,2$ and the kernel of 
	\[ H^1_{\tau}(X,\mathcal F)\to H^1_{\tau}(X\times G,\mathcal F),\]
	which is precisely $ H^1_{\tau}(X,\mathcal F)^G$ if the displayed injectivity condition holds.
	
	It remains to check that \eqref{eq:CL-condition} holds in the given examples: It suffices to show this for $i=1$ and for $X$ in the basis of affinoid perfectoid spaces in $Y_{\tau}$. But here we have 
	\[\O(X\times G)=\O(G)\hotimes_K \O(X)=\Map_{\cts}(G,K)\hotimes_K \O(X)=\Map_{\cts}(G,\O(X)).\]
	The case of $M_n(\O)$ follows by forming products, the case of $\GL_n(\O)$ by taking units.
\end{proof}
As an immediate application, this tells us that continuous $1$-cocycles are precisely the descent data for $X\to Y$ on the trivial vector bundle $\O^n$ on $X$:
\begin{Corollary}\label{c:Cartan--Leray-for-line-bundles}
	Let $X\to Y$ be Galois with group $G$, then there is a left-exact sequence
	\[ 0\to H^1_{\cts}(G,\GL_n(\O(X)))\to H^1_v(Y,\GL_n)\to H^1_v(X,\GL_n)^G.\]
	More functorially, this is given by sending any continuous $1$-cocycle $c:G\to \GL_n(\O(X))$ to the $v$-vector bundle
	$V$ on $Y$ defined on $Y'\in Y_v$ by
	\[ V(Y')=\{ x\in \O^n(Y'\times_YX)\mid g^\ast x=c(g)x \text{ for all }g\in G\}.\]
\end{Corollary}
\begin{proof}
	The first part follows from Proposition~\ref{p:Cartan--Leray}, the last one from Lemma~\ref{l:effective-descent}
\end{proof}

\subsection{The sheaf of principal units}
In this section, let $X$ be either a smooth rigid space over $K$ or a perfectoid space over $K$. We consider the (big) site $X_{\tau}$ for $\tau$ one of the following topologies: The \'etale or pro-\'etale topology from \cite[\S2.1]{huber2013etale} and \cite[Definition 3.9]{Scholze_p-adicHodgeForRigid} if $X$ is rigid, or the \'etale, pro-\'etale or $v$-topology from \cite[Definition 8.1]{etale-cohomology-of-diamonds} if $X$ is perfectoid. In particular, $\Perf_{K,\tau}={\Spa(K)}_{\tau}$.
\begin{Definition}
	We denote by $U_{\tau}:=1+\m\O_{\tau}^+\subseteq \O_{\tau}^\times$ the subsheaf of $\O_{\tau}^\times$ of principal units. This is represented (in diamonds over $K$) by the open disc of radius $1$ centred at $1\in \G_m$. It contains the sheaf of $p$-power unit roots $\mu_{p^\infty}\subseteq U_{\tau}$, but not all unit roots $\mu\subseteq \O_{\tau}^\times$.
\end{Definition}
The following sheaf will be very useful to compute Picard groups of diamonds: Roughly, it plays the same role in determining the cohomology of $\O^\times_{\tau}$ as the sheaf $\O_{\tau}^+/p$ has for $\O^+_{\tau}$.
\begin{Definition}		
	We denote by $\overline{\O}_{\tau}^\times$ the abelian sheaf on $X_{\tau}$ defined as the quotient
	\[\overline{\O}_{\tau}^\times:=\O_{\tau}^{\times}/U_{\tau}=\O_{\tau}^{\times}/(1+\m\O^+_\tau).\] 
	We will often simply denote the sheaf $\overline{\O}_v^\times$  on $\Perf_{K,v}$ by $\overline{\O}^\times$.
\end{Definition}

\begin{Definition}
	Let $G$ be a topological abelian group, written multiplicatively. Following \cite[\S3]{Robertson_toptorsion}, we call an element $x\in G$ a topological torsion element if 
	\[x^{n!}\to 1\quad \text{for } n\to \infty.\]
	In all situations that we will encounter, this will be equivalent to the condition that there is $N\in \N$ for which $x^{Np^n}\to 1$ for $n\to \infty$. For example, the topological torsion subgroup of $K^\times$ is $(1+\m_K)\mu(K)$ where $\mu(K)\subseteq K^\times$ is the subgroup of all unit roots.
\end{Definition}
\begin{Definition}\label{d:top-torsion-subsheaf}
	We denote by $\O^{\times,\mathrm{tt}}\subseteq \O^\times$ the topologically torsion subsheaf. Explicitly, this is the subsheaf generated by $U=1+\m\O^+$ and the subsheaf $\mu$ of unit roots.
\end{Definition}
\begin{Definition}\label{d:commute-tf}
For multiplicative sheaves like $\bOx$, we write $\bOx\tf$ for the sheaf $\varinjlim_{x\mapsto x^p}\bOx$  obtained by inverting $p$ on the sheaf of abelian groups. We caution that this involves a sheafification, so we do not in general have $\O^\times\tf(X)=\O^\times(X)\tf$ (e.g.\ not for $X=\G_m$). However, this holds on quasi-compact objects, like affinoids in any of the sites we consider.
\end{Definition}
\begin{Lemma}\label{l:bOx-p-divisible}
	\begin{enumerate}
	\item We have $\overline{\O}_{\tau}^\times\tf=\overline{\O}_{\tau}^\times$, i.e.\ the sheaf $\overline{\O}_{\tau}^\times$ is uniquely $p$-divisible.
	\item We have $(\O^\times_\tau/\O_\tau^{\times,\mathrm{tt}})\otimes_{\Z}\Q=\O^\times_\tau/\O_\tau^{\times,\mathrm{tt}}$, i.e.\ the sheaf $\O^\times_\tau/\O_\tau^{\times,\mathrm{tt}}$ is uniquely divisible.
	\end{enumerate}
\end{Lemma}
\begin{proof}
	This follows from the commutative diagram of exact sequences in the \'etale topology
	\[\begin{tikzcd}
		1 \arrow[r] & \mu_{p} \arrow[r] & \O_{\tau}^\times \arrow[r, "p"] & \O_{\tau}^\times \arrow[r] & 1 \\
		1 \arrow[r] & \mu_{p} \arrow[r] \arrow[u,equal] & U_{\tau} \arrow[r, "p"] \arrow[u] & U_{\tau} \arrow[u] \arrow[r] & 1.
	\end{tikzcd}\]
	The second part follows from the same argument for the exact sequence
	\[ 1\to \mu_N\to \O^{\times,\mathrm{tt}}\xrightarrow{N} \O^{\times,\mathrm{tt}}\to 1.\qedhere\]
\end{proof}
Our interest in $\overline{\O}^\times$ stems from the following key approximation lemma, which says that in contrast to $\O^\times_{\proet}$, the sheaf $\overline{\O}^\times_{\proet}$  arises via pullback from the \'etale site. 
\begin{Lemma}\label{l:limit-of-bOx}
	Let $X$ be a smooth rigid space over $K$.
	Let $X_\infty$ be an affinoid perfectoid object in $X_{\proet}$ that can be represented as $X_\infty=\varprojlim_{i\in I} X_i$  for some affinoids $X_i$.
	Then 
	\[\overline{\O}_{\proet}^{\times}(X_\infty)=\varinjlim_{i\in I} \bOx_{\et}(X_i).\]
	In particular, for the morphism of sites $u:X_{\proet}\to X_{\et}$, we have \[\bOx_{\proet}=u^{\ast}\bOx_{\et}.\]
	Similarly, we have $\O^\times_{\proet}/\O^{\times,\mathrm{tt}}_{\proet}=u^{\ast}(\O^\times_{\et}/\O^{\times,\mathrm{tt}}_{\et})$.
\end{Lemma}
For the proof, we crucially use that we work in the ``flattened pro-\'etale site'' of \cite{Scholze_p-adicHodgeForRigid}, rather than finer variants. We also need the $p$-adic logarithm sequence, which we now recall.

\subsection{The $p$-adic exponential and its higher direct image}
In complex geometry, a useful tool to study line bundles is the exponential exact sequence
\[ 0\to 2\pi i\Z\to \O\xrightarrow{\exp} \O^\times\to 0.\]
Over $\Q_p$, we have the following analogue of this sequence:

\begin{Lemma}\label{l:p-adic-exponential}
	Let $p'=p$ if $p>2$ and $p'=4$ if $p=2$. The $p$-adic exponential and logarithm map define homomorphisms of rigid group varieties
	\[\exp: p'\G_a^+\to 1+p'\G_a^+, \]
	\[ \log:1+\m \G_a^+\to \G_a\]
	such  that $\log(1+p'\G_a^+)\subseteq p'\G_a^+$ and $\exp\circ \log=\id$ on $1+p'\G_a^+$ and $\log\circ\exp=\id$ on $p'\G_a^+$.

	In particular, the logarithm defines a short exact sequence of sheaves
	\begin{equation}\label{eq:log}
		1\to \mu_{p^\infty}\to U_{\tau}\xrightarrow{\log}  \O_{\tau}\to 1,
	\end{equation}
	whereas the exponential defines a short exact sequence
	\begin{equation}\label{eq:exp}
		1\to \O_{\tau}\xrightarrow{\exp}  \O^\times_{\tau}\tf\to  \bOx_{\tau}\to 1.
	\end{equation}
\end{Lemma}
\begin{proof}
	The first sequence is well-known, see e.g.\ \cite[\S7]{dJ_etalefundamentalgroups}. We sketch the argument:
	
Clearly $\log(x)=\sum (-1)^{n}(x-1)^n/n$ and $\exp(x)=\sum x^n/n!$ define rigid analytic maps over $\Q_p$ as described. By classical non-archimedean analysis, these have the desired properties on $\C_p$-points. It follows that they also hold on the level of rigid groups. 
	
	To get the first exact sequence, one observes that the kernel of $\log$ has to be $\mu_{p^\infty}$ since for any $x\in U$, some power $x^{p^n}$ lies in $1+p'\O^+$ where $\log$ is injective. The logarithm is surjective in the \'etale topology since for any $x\in \O$ with $p^{n}x\in p'\O^+$, any $p^{n}$-th root $y$ of the unit $\exp(p^{n}x)$, which exists \'etale-locally, will satisfy $\log(y)=\tfrac{1}{p^{n}}\log(\exp(p^{n}x))=x$.
	
	For the exponential sequence, consider the short exact sequence (we omit $\tau$) 
	\[0\to p'\O^{+}\xrightarrow{\exp}\O^{\times}\to \O^{\times}/(1+p'\O^+)\to 1.\]
	After inverting $p$, this becomes the exact sequence \eqref{eq:exp}: This is because $(1+\m\O^+)/(1+p'\O^+)$ is $p^\infty$-torsion, and thus $\O^{\times}/(1+p'\O^+)\tf=\bOx\tf=\bOx$ by Lemma~\ref{l:bOx-p-divisible}.1.
\end{proof}

As an immediate consequence, we get an explicit description of $\bOx$ on a basis of $X_{\tau}$:
\begin{Lemma}\label{l:bOx-explicit}
	Let $Y$ be a quasi-compact object of $X_{\tau}$ such that $H^1_\tau(Y,\O)=0$. Then
	\[\bOx_{\tau}(Y)=(\O_{\tau}^\times(Y)/U_{\tau}(Y))\tf.\]
\end{Lemma}
\begin{proof}
 We evaluate \eqref{eq:exp} at $Y$ and commute $\tf$ with $H^0(Y,-)$ like in Definition~\ref{d:commute-tf}.
\end{proof}
We now use this to prove the key lemma from the last subsection:
\begin{proof}[Proof of Lemma~\ref{l:limit-of-bOx}]
	It suffices to prove this locally on an analytic cover of $X_\infty$, so we may assume that the map
	\[\phi:\varinjlim\O(X_i)\to \O(X_\infty)\]
	has dense image. We claim that in this case the map
	\begin{equation}\label{eq:tilde-limit-property-on-O^x}
		\phi:\varinjlim\O^\times(X_i)\to \O^\times(X_\infty)
	\end{equation}
  has dense image, too. To see this, let $f\in \O^\times(X_\infty)$, and let $\phi(f_i)\to f$ with $f_i\in \O(X_i)$ be any converging sequence in the image, and similarly $\phi(f'_i)\to f^{-1}$, then we have
	\[\phi(f_if'_i)\to 1.\]
	In particular, for $i$ large enough, we have $\phi(f_if'_i)\in 1+\m\O^+(X_\infty)=U(X_\infty)$. 
	\begin{claim}\label{eq:O^+(U)/p -> O^+(U_infty)/p injective}
		For $i\gg0$, we have 
		\[\O(X_i)\cap \phi^{-1}(1+\m\O^+(X_\infty))=1+\m\O^+(X_i).\]
	\end{claim}
	\begin{proof}
		The inclusion ``$\supseteq$'' is clear. To see the other, recall that $f\in \O(X_i)$ is in $\O^+(X_i)$ if and only if $|f(x)|\leq 1$ for all $x\in X_i$. Since $X_\infty\to X_i$ is surjective on the underlying topological spaces for $i\gg0$, this can be checked after pullback to $X_\infty$.
	\end{proof}
	This implies that 
	\[f_if'_i\in 1+\m\O^+(X_i)\subseteq \O^{\times}(X_i)\]
	for $i\gg 0$, and thus $f_i\in \O^\times(X_i)$, as desired.
	
	We conclude from combining~\eqref{eq:tilde-limit-property-on-O^x} and Claim~\ref{eq:O^+(U)/p -> O^+(U_infty)/p injective} that the induced map
	\[\varinjlim \O^\times(X_i)/U(X_i) \to  \O^\times(X_\infty)/U(X_\infty)\]
	is an isomorphism.
	Since the $X_i$ are affinoid and $X_\infty$ is affinoid perfectoid, it follows from Lemma~\ref{l:bOx-explicit} applied to the \'etale site on the left and the pro-\'etale site on the right that also
	\[ \varinjlim \bOx_{\et}(X_i)\isomarrow \bOx_{\proet}(X_\infty) \]
	is an isomorphism. This proves the first part. The second follows from \cite[Lemma~3.16]{Scholze_p-adicHodgeForRigid}.
	
	The case of $\O^\times/\O^{\times,\mathrm{tt}}$ follows since by Lemma~\ref{l:bOx-p-divisible}, we have $\O^\times/\O^{\times,\mathrm{tt}}=\bOx\otimes_\Z\Q$.
\end{proof}

We now use this to prove the main result of this section:
\begin{Proposition}\label{p:exponential-on-R^1O->R^1O^x}
	Let $X$ be a smooth rigid space over $K$.
	Then for the morphism of sites $\nu:X_{v}\to X_{\et}$, the short exact sequences \eqref{eq:log} and \eqref{eq:exp} induce natural isomorphisms
	\begin{alignat*}{3}
		\log\colon&R^i\nu_{\ast}U&\isomarrow& R^i\nu_{\ast}\O \quad& \text{ for any }i\geq 1,\\
		\exp\colon&R^1\nu_{\ast}\O&\isomarrow& R^1\nu_{\ast}\O^\times.&
	\end{alignat*}
\end{Proposition}

For the proof, we use Lemma~\ref{l:limit-of-bOx} as a stepping stone to get to the $v$-topology:
\begin{Lemma}\label{l:R^1nu_*bOx}
	In the setting of Proposition~\ref{p:exponential-on-R^1O->R^1O^x}, we have
	\begin{enumerate}
		\item $\displaystyle\nu_{\ast}\bOx_{v}=\bOx_{\et}$,
		\item $R^1\nu_{\ast}\bOx_{v}=0$.
	\end{enumerate}
\end{Lemma}
\begin{proof}
	We can split up $\nu$ into the two morphism of sites
	\[  \nu_{\ast}:X_{v}\xrightarrow{w} X_{\proet}\xrightarrow{u} X_{\et}.\]
	As $\O_v$ and $\O_{\proet}$ are both acyclic on affinoid perfectoids, we know that \[Rw_{\ast}\O_{v}=\O_{\proet}.\]
	Commuting $w_{\ast}$ and $\tf$ like in Definition~\ref{d:commute-tf},
	we also have \[w_{\ast}(\O_{v}^\times\tf)=\O_{\proet}^\times\tf.\]
	By the long exact sequence of $w_{\ast}$ for \eqref{eq:exp}, this together implies that \[w_\ast\bOx_{v}=\bOx_{\proet}.\]
	Similarly, since any $v$-topological line bundle on an affinoid perfectoid space is trivial in the analytic topology by Theorem~\ref{t:KL-line-bundles}, and affinoid perfectoids form a basis of $X_{\proet}$, we have 
	\[R^1w_{\ast}(\O^\times_{v}\tf)=0.\]
	It follows from $R^2w_{\ast}\O=0$ that
	\[\quad R^1w_{\ast}\bOx_{v}=0.\]
	We now combine these to get to $\nu$:
 	By the Leray spectral sequence, the above implies 
	\[ R^1\nu_{\ast}\bOx_{v} = R^1u_{\ast}(w_{\ast}\bOx_{v})=R^1u_{\ast}\bOx_{\proet}.\]
	We have thus reduced to considering $u:X_{\proet}\to X_{\et}$.
	Here we have $\bOx_{\proet}=u^{\ast}\bOx_{\et}$  by Lemma~\ref{l:limit-of-bOx},
	which by \cite[Corollary~3.17.(i)]{Scholze_p-adicHodgeForRigid} implies $\bOx_{\et}=u_{\ast}\bOx_{\proet}$
	as well as
	\[R^1u_{\ast}\bOx_{\proet}=0.\]
	Putting everything together, this proves the lemma.
\end{proof}
\begin{Lemma}\label{l:Rvmu=0}
	Let $Y$ be any diamond, then for $\nu:Y_v\to Y_{\et}$ we have $R\nu_{\ast}\mu_{p^\infty}=\mu_{p^\infty}$.
\end{Lemma}
\begin{proof}
	Since $\mu_{p^\infty}$ is a constant sheaf, this follows from \cite[Propositions 14.7, 14.8]{etale-cohomology-of-diamonds}.
\end{proof}

We now have everything in place to prove Proposition~\ref{p:exponential-on-R^1O->R^1O^x}:
\begin{proof}[Proof of Proposition~\ref{p:exponential-on-R^1O->R^1O^x}]
	The first part follows from Lemma~\ref{l:Rvmu=0} and the sequence \eqref{eq:log}.
	
	For the second isomorphism, consider the long exact sequence  of the exponential \eqref{eq:exp}
	\[0\to \nu_{\ast}\O \to \nu_{\ast}(\O^\times\tf)\to \nu_{\ast}\bOx\to R^1\nu_{\ast}\O\xrightarrow{\exp} R^1\nu_{\ast}(\O^\times\tf) \to R^1\nu_{\ast}\bOx\tf.\]
	 By Lemma~\ref{l:R^1nu_*bOx}.1, we have $\nu_{\ast}\bOx=\bOx_{\et}$. As we have a map $\O_{\et}^\times\tf\to \nu_{\ast}(\O^\times\tf)$ (in fact an isomorphism by Remark~\ref{r:nu_*O} below, but we do not need this here), this shows that the boundary map vanishes. Thus $\exp$ in the sequence is injective. The last term vanishes by Lemma~\ref{l:R^1nu_*bOx}.2, hence $\exp$ is an isomorphism. 
	 Finally, the colimit of the Kummer sequence
	 \[ 1\to \mu_{p^\infty}\to \O^\times\to \O^\times\tf\to 1\]
	 combines with Lemma~\ref{l:Rvmu=0} to show that we have $R^i\nu_{\ast}\O^\times=R^i\nu_{\ast}\O^\times\tf$ for any $i\geq 1$.
\end{proof}

With these preparations, we can now deduce the first part of Theorem~\ref{t:Picv-ses}, using a variant of a result of Scholze describing $R^1\nu_{\ast}\O$.
\begin{Definition}\label{d:BK-twist}
	Let $\theta:W(\O_{K^\flat})\to \O_K$ be Fontaine's map.
	For any $i\in \Z$, we denote by $\O_K\{i\}:=(\ker \theta)^i/(\ker\theta)^{i+1}$ the $i$-th Breuil--Kisin--Fargues twist. This is non-canonically isomorphic to $\O_K$ as an $\O_K$-module. For any $\O_K$-module $M$, or a sheaf of such, we set
	\[ M\{i\}:=M\otimes_{\O_K}\O_K\{i\}.\]
	As explained in \cite[Example~4.24]{BMS}, if $K$ contains all $p$-power unit roots, then there is a canonical isomorphism
	\[ K\{i\}=K(i)\]
	where the right-hand-side denotes the Tate twist $K(i)=K\otimes_{\Z}\Z_p(i)$. In this sense, Breuil--Kisin--Fargues twists are a generalisation of Tate twists to general perfectoid base fields.
	
	Finally, we set 
	\[ \wt\Omega^i_X:=\Omega^i_X\{-i\}.\]
\end{Definition}
\begin{Proposition}[{\cite[Proposition 3.23]{Scholze2012Survey}}]\label{p:Hodge-Tate-spectral-sequence}
	Let $K$ be a perfectoid field extension of $\Q_p$ and let $X$ be a smooth rigid space over $K$.
	Let $\nu:X_{v}\to X_{\et}$ be the natural morphism of sites. Then there are canonical isomorphisms  on $X_{\et}$ for all $i\geq 0$:
	\[ R^i\nu_{\ast}\O= \wt\Omega^i_{X}=\Omega^i_X\{-i\}.\]
\end{Proposition}
\begin{Remark}\label{r:nu_*O}
	Already for $i=0$, this is the non-trivial result that $\nu_{\ast}\O=\O$, proved more generally by Kedlaya--Liu for semi-normal rigid spaces \cite[Theorem 8.2.3]{KedlayaLiu-II}.
	
	The notation $ \wt\Omega^i_X$ is motivated by \cite[\S8]{BMS}, where a much finer integral result about $\O^+$ is proved for $X$ that admit a smooth formal model.
\end{Remark}
\begin{proof}
	For algebraically closed $K$, this is shown in
	\cite[Proposition 3.23]{Scholze2012Survey} for $X_{\proet}\to X_{\et}$.
	But for  $w:X_{v}\to X_{\proet}$, we have $Rw_{\ast}\O=\O_{\proet}$, so the case of $\nu$ follows.

	The case of general perfectoid $K$ follows from this by Galois descent, by an argument similar to that of \cite[Proposition 6.16.(ii)]{Scholze_p-adicHodgeForRigid}. Since we do not know a reference for this in the literature in the desired generality, we sketch a proof here: Recall that $C$ is the completion of an algebraic closure of $K$. It suffices to prove that for any smooth affinoid rigid space $X=\Spa(A)$ over $K$ that is standard-\'etale over a torus $\mathbb T^d$, we have a natural isomorphism
	\[ H^j_v(X,\O)=H^0(X,\wt\Omega^j_{X}).\]
	To see this, let $X_C=\Spa(A_C)$ and let $\wt X=\Spa(\wt A)$ be the pullback along the toric tower $\wt {\mathbb T}^d\to \mathbb T^d$. Let $\wt X_C=\Spa(\wt A_C)$, then we have a Cartesian square of pro-\'etale covers in $X_{\proet}$
	\[ \begin{tikzcd}
		\wt X \arrow[d] & \wt X_C \arrow[d] \arrow[l] \\
		X & X_C \arrow[l]
	\end{tikzcd}\]
	in which the horizontal maps are Galois with group $G:=\Gal(C|K)$ and the map on the right is Galois with group $\Z_p^d(1)$. Since $\wt X$ and $\wt X_C$ are each affinoid perfectoid, $\O$ is acyclic on them. The Cartan--Leray sequence of Corollary~\ref{c:CL-acyclic} for the right map therefore shows
	\[ H^j_{\cts}(\Z_p^d(1),\wt A_C)=H^j_v(X_C,\O)=H^0(X_C,\wt\Omega^j_{X}).\]
	by the first part. More generally, for any $n\geq 0$, the same  Cartan--Leray sequence for $\wt X_C\times G^n\to X_C\times G^n$ combines with \cite[Lemma~5.5]{Scholze_p-adicHodgeForRigid} to show that for $n\geq 0$, we have
	\[ H^i_v(X_C\times G^n,\O)=H^i_{\cts}(\Z_p^d(1),\O(\wt X_C\times G^n))=\Map_{\cts}(G^n,H^0(X_C,\wt\Omega^j_{X})).\]
	We thus get the full Cartan--Leray spectral sequence from Proposition~\ref{p:Cartan--Leray}.\ref{i:CL-full}:
	\[ H^i_{\cts}(G,H^0(X_C,\wt\Omega^j_{X}))\Rightarrow H^{i+j}_v(X,\O).\]
	The \'etale map $X\to \mathbb T^d$ induces an isomorphism $\Omega^j_{X}\cong\O_X^k$ where $k={d\choose j}$. Consequently,
	\[H^0(X_C,\wt\Omega^j_{X})=H^0(X,\wt\Omega^j_{X})\hotimes_{K}C=A^k\{-j\}\hotimes_KC\cong A^k\hotimes_KC=A_C^k\]
	as topological $G$-modules.
	We claim that $H^i_{\cts}(G,A_C)=0$ for $i\geq 1$. Indeed, observe that the map $A_C\to \wt A_C$ is split in topological $G$-modules: This is because by \cite[Lemma~4.5]{Scholze_p-adicHodgeForRigid}, we can pullback the canonical module-splitting of $\wt {\mathbb T}^d\to \mathbb T^d$. We thus have an injection
	\[H^i_{\cts}(G,A_C)\hookrightarrow H^i_{\cts}(G,\wt A_C)=H^i_v(\wt X,\O)=0\]
 	by Corollary~\ref{c:CL-acyclic} this time applied to the top map using that $\wt X$ is affinoid perfectoid.
	
	All in all, this shows that the above spectral sequence collapses and induces isomorphisms
	\[ H^{j}_v(X,\O)=H^0_{\cts}(G,H^0(X_C,\wt\Omega^j_{X}))=(\wt\Omega_X^j(X)\hotimes_KC)^G=\wt\Omega_X^j(X).\qedhere\]
\end{proof}

\begin{Definition}
	We denote by $\HT$ the induced map in the Leray sequence
	\[ \HT:H^1_v(X,\O)\to H^0(X,\Omega_X^1\{-1\}).\]
\end{Definition}
Combining Propositions~\ref{p:exponential-on-R^1O->R^1O^x} and \ref{p:Hodge-Tate-spectral-sequence}, we see:
\begin{Corollary}\label{c:R1nu_*O^x}
	The logarithm induces an isomorphism $\HT\log:R^1\nu_{\ast}\O^\times=\Omega^{1}_X\{-1\}$.
\end{Corollary}
This shows the first part of Theorem~\ref{t:Picv-ses}: In fact, it implies the following stronger form which also bounds the cokernel on the right in terms of the Brauer group of $X$:
\begin{Theorem}\label{t:5TES-version}
	Let $X$ be a smooth rigid space over $K$. Then the $5$-term exact sequence of the Leray spectral sequence of  $\nu:X_{v}\to X_{\et}$ for the sheaf $\O^\times$ is of the form
	\[ 0\to \Pic_{\an}(X)\to \Pic_{v}(X)\xrightarrow{\HT\log} H^0(X,\widetilde\Omega_X^1)\to H^2_{\et}(X,\O^\times)\to  H^2_{v}(X,\O^\times).\]
\end{Theorem}
\begin{proof}
	We consider the 5-term exact sequence of the Leray sequence for $\nu:X_v\to X_{\et}$. By Remark~\ref{r:nu_*O}, we have $\nu_{\ast}\O^\times=\O^\times_{\et}$, so its first term is $\Pic_{\et}(X)$. This is equal to $\Pic_{\an}(X)$ by  \cite[Proposition~8.2.3]{FvdP}.
	By Corollary~\ref{c:R1nu_*O^x}, the third term is as described.
\end{proof}
\begin{Remark}
	In \cite{heuer-Picard-good-reduction} it is shown that Lemma~\ref{l:R^1nu_*bOx} generalises and we in fact have $R\nu_{\ast}\bOx=\bOx$. It follows that for any $i\geq 1$, the exponential induces isomorphisms 
	$R^i\nu_{\ast}\O^\times=R^i\nu_{\ast}\O=\Omega^i_X\{-i\}$.
	This gives a ``multiplicative Hodge--Tate spectral sequence'' relating e.g.\ the \'etale to the $v$-topological Brauer group in terms of Hodge cohomology. 
\end{Remark}

\section{Analyticity criteria}
As a first application, we now deduce from Theorem~\ref{t:Picv-ses}.1 some criteria for deciding whether a given $v$-line bundle is analytic. We think that these will be useful in practice (for instance, see Example~\ref{exmp:CHJ}).
We start with a direct consequence of exactness of the $\HT\log$-sequence:

\begin{Corollary}\label{c:analytic-on-Zar-dense-implies-analytic}
	Let $X$ be a smooth rigid space and $L$ a $v$-line bundle on $X$. Let $V\subseteq X$ be any Zariski-dense open subspace. Then $L$ is analytic if $L|_{V}$ is. More generally, let $f:Y\to X$ be a smooth morphism with Zariski-dense image. Then $L$ is analytic if and only if $f^{\ast}L$ is.
\end{Corollary}
\begin{proof}
	By Theorem~\ref{t:Picv-ses}.1, $L$ is analytic if and only if $\HT\log(L)=0$. As we can check this locally, we may assume that $X$ and $Y$ are affinoid. Then since $f$ is smooth with Zariski-dense image, the map $H^0(X,\Omega^1)\hookrightarrow H^0(Y,\Omega^1)$ is injective. Now use that $\HT\log$ is functorial.
\end{proof}
Second, we can use this to give a characterisation in terms of non-trivial sections:
\begin{Proposition}\label{p:non-trivial-section-implies-analytic}
	Let $X$ be a smooth connected rigid space and $L$ a $v$-line bundle on $X$. If $H^0(X,L)\neq 0$, then $L$ is analytic. 
\end{Proposition}
\begin{proof}
	The statement is local on $X$, so we can assume that $X$ is affinoid and \'etale over a torus. In particular, there is a toric pro-\'etale affinoid perfectoid Galois cover
	\[ X_\infty\to \dots \to X_1\to X_0=X\]
	with Galois group $G$. By Corollary~\ref{c:analytic-on-Zar-dense-implies-analytic}, it suffices to prove that $L$ becomes trivial on \textit{some} non-empty open subspace $V\subseteq X_n$ for some $n$.
	
	Since $L$ is trivial analytic-locally on $X_\infty$ by Theorem~\ref{t:KL-line-bundles}, we can after passing to some affinoid open $V_n\subseteq X_n$ assume that it is trivial on $X_\infty$. In this case, we know by Corollary~\ref{c:Cartan--Leray-for-line-bundles} that $L$ is associated to a $1$-cocycle
	$c:G\to \O^\times(X_\infty)$, that is we have
	\[ H^0(X,L)=\{ f\in \O(X_\infty)\mid g^\ast f=c(g)f  \quad \text{ for all }g\in G\}.\]
	
	Assume now that we have a non-trivial element $f\in H^0(X,L)$. We claim that this is invertible on the pullback of some open $V\subseteq X_n$ for some $n$. To see this, we use:
	\begin{claim}
		There is $x\in X_\infty(C,\O_C)$ with $f(x)\neq 0$.
	\end{claim}
	\begin{proof}
		Write $X_i=\Spa(A_i,A_i^+)$ for $i\in \N$, then by \cite[Lemma~4.5]{Scholze_p-adicHodgeForRigid} we have $X_\infty=\Spa(A,A^+)$ where $A=(\varinjlim A_i^+)^\wedge\tf$. 	We have compatible maps for each $i\in \N$ and $k\in \N$
		\[ A_i^+/p^k\hookrightarrow \Map_{\lc}(X_i(C,\O_C), \O_C/p^k),\quad f\mapsto (x\mapsto f(x))\]
		which are injective by the Maximum Modulus Principle since $A_i$ is smooth, so $A_{I,C}$ is reduced, and $C$ is algebraically closed. In the colimit over $i$, we obtain an injection
		\[ \textstyle\varinjlim_{i\in I}A_i^+/p^k\hookrightarrow \textstyle\varinjlim_{i\in I}\Map_{\lc}(X_i(C), \O_C/p^k)\hookrightarrow \Map_{\lc}(X_\infty(C), \O_C/p^k).\]
		Taking the inverse limit over $k$ and inverting $p$, we get an injection
		\[ \O(X_\infty) \hookrightarrow \Map_{\cts}(X_\infty(C),C).\]
		This gives the desired statement.
	\end{proof}
	We deduce from the claim that there is $k\in \N$ such that $|f(x)|\geq |\varpi^k|$. Consequently, the rational open $V_\infty$ of $X_\infty$ defined by $|f|\geq  |\varpi^k|$ is non-empty. We can therefore find a non-empty rational open $V$ in some $X_n$ whose pullback to $X_\infty$ is contained in $U_\infty$. We replace $V_\infty$ by this pullback, then in particular, $f$ is invertible on $V_\infty$.
	
	But if $f\in \O^\times(V_\infty)$, then multiplication by $f$ defines an isomorphism $\O_{|V}\isomarrow L_{|V}$. In particular, $L$ is trivial on $V$, in particular analytic, and thus it is analytic on $X$.
\end{proof}
Combining Corollary~\ref{c:analytic-on-Zar-dense-implies-analytic} and Proposition~\ref{p:non-trivial-section-implies-analytic}, we deduce a stronger version:
\begin{Corollary}
	Let $X$ be a smooth connected rigid space. Then a $v$-line bundle $L$ on $X$ is analytic if and only if $\nu_{\ast}L\neq 0$, where $\nu:X_{v}\to X_{\et}$ is the natural morphism of sites.
\end{Corollary} 

\begin{Corollary}\label{c:map-from-v-line-bundle}
	Let $X$ be a smooth connected rigid space.
	Let $V$ be an analytic vector bundle and $L$ a $v$-line bundle on $X$. If there is a non-trivial map $L\to V$, then $L$ is analytic.
\end{Corollary}
\begin{proof}
	The statement is local on $X$, so we can assume that $V=\O^n$ is trivial. Thus $f:L\to \O^n$ consists of functions $f_i:L\to \O$ for $i=1,\dots,n$ and $f$ is non-trivial if one of the $f_i$ is. We are thus reduced to $V=\O$. But then $f\neq 0$ if and only if its dual $f^\vee:\O\to L^\vee$ is non-trivial. By Proposition~\ref{p:non-trivial-section-implies-analytic}, this implies that $L^\vee$ is analytic, and thus so is $L$.
\end{proof}
Third, the property of "being analytic" on products can be checked on fibres:
\begin{Corollary}\label{c:analytic-on-products}
	Let $X$ and $Y$ be a smooth rigid spaces and let $L$ be a $v$-line bundle on $X\times Y$. Assume that there are Zariski-dense sets of points $S\subseteq X(K)$ and $T\subseteq Y(K)$ such that $L_x$ on $Y$ for $x\in S$ and $L_y$ on $X$ for $y\in T$ are all analytic. Then $L$ is analytic.
\end{Corollary}
\begin{proof}
	We can without loss of generality assume that $X$ and $Y$ are affinoid. As they are smooth, $\Omega^1_X$ and $\Omega^1_Y$ are vector bundles, and after localising further we may assume that they are free, with bases $v_1,\dots v_n$ and $w_1,\dots w_m$, respectively. We then have
	\[ \Omega^1(X\times Y)=(\O(X\times Y)\otimes_{\O(X)}\Omega_X^1(X))\oplus (\O(X\times Y)\otimes_{\O(Y)}\Omega_Y^1(Y)),\]
	this follows from the corresponding algebraic statement for finitely generated $K$-algebras.
	
	The corollary now follows from Theorem~\ref{t:Picv-ses}.1: According to the above decomposition,
	\[ \HT\log(L)=\sum_{i=1}^n f_i\otimes v_i +\sum_{j=1}^m g_j\otimes w_j.\]
	Then $\HT\log(L_x)=\sum g_j(x)w_j\in \wt\Omega^1(Y)$. If this vanishes for all $x\in S$, then all $g_j$ vanish on $S\times Y$, and thus also on its Zariski-closure $X\times Y$. Thus $g_j=0$. Similarly for $y\in T$.
\end{proof}

Finally, we note that if we add good reduction assumptions, then any $v$-line bundle trivialised by a Galois cover of good reduction is already trivial in the Zariski topology. 

\begin{Definition}\label{d:pro-etale-torsor-of-formal-schemes}
	Let $\mathfrak X$ be a formal scheme of topologically finite type over $\O_K$. Let $G$ be a pro-finite group. We say that a morphism $\mathfrak X_\infty\to \mathfrak X$ is a pro-\'etale $G$-torsor with group $G$  if there is a cofiltered inverse system $(\mathfrak X_i)_{i\in I}$ of finite \'etale Galois torsors $\mathfrak X_i\to \mathfrak X$ of group $G_i$ such that $\mathfrak X_\infty=\varprojlim \mathfrak X_i$ and $G=\varprojlim G_i$. Then $\mathfrak X$ automatically has a $G$-action.
\end{Definition}
\begin{Proposition}\label{p:pro-etale-formal-model-implies-analytic}
	Let $\mathfrak X$ be a formal scheme of topologically finite type over $\O_K$ and let $\mathfrak X_\infty\to \mathfrak X$ be a pro-\'etale $G$-torsor. Then for any $1$-cocycle $c:G\to \mathcal O(\mathfrak X_\infty)^\times$, the associated $v$-line bundle on the generic fibre $\mathcal X$ is the analytification of a Zariski-line bundle on $\mathfrak X$.
\end{Proposition}
\begin{proof}
	With notation as in Definition~\ref{d:pro-etale-torsor-of-formal-schemes},
	the statement is local on $\mathfrak X_0:=\mathfrak X$, so we can reduce to the case that $\mathfrak X$ is affine, and thus so are the $\mathfrak X_i=\Spf(A_i)$ as well as $\mathfrak X_\infty=\Spf(A_\infty)$. 
	
	Let $\mathcal X_i$ be the rigid generic fibre of $\mathfrak X_i$ and let $\mathcal X_\infty=\varprojlim \mathcal X_i$ as a diamond. Since $\mathfrak X_i$ is affine, we then have a natural $G$-equivariant morphism of $\O_K$-algebras
	\[ \O(\mathfrak X_\infty)=\varprojlim_n\varinjlim_i\O(\mathfrak X_i)/p^n\to  \varprojlim_n\varinjlim_i\O^+/p^n(\mathcal X_i)\to\varprojlim_n\O^+/p^n(\mathcal X_\infty)=\O^+(\mathcal X_\infty).\]
	Using Corollary~\ref{c:Cartan--Leray-for-line-bundles}, we thus indeed get a $v$-line bundle $L$ on $\mathcal X$ associated to $c$. Furthermore, by Corollary~\ref{c:Cartan--Leray-for-line-bundles}, this $L$ is given on any  $Y\to \mathcal X$ in $\mathcal X_v$ by
	\[ L(Y)=\{f\in \O(Y\times_{\mathcal X_0} \mathcal X_\infty)|g^{\ast}f=c(g)f \text{ for all }g\in G\}.\]
	It thus suffices to prove that Zariski-locally on $\mathfrak X_0$, there is $f\in \O(\mathfrak X_\infty)^\times$ such that $g^{\ast}f=c(g)f$, since then $L|_Y=f\O|_Y$, which shows that $L$ is trivial on $Y$.
	
	Consider now for each $n\in \N$ the  reduction of the cocycle $c$ mod $p^n$
	\[G\xrightarrow{c} A_\infty^\times\to (A_\infty/p^n)^\times.\]
	As this factors over a finite quotient of $G$ \cite[(1.2.5) Proposition]{NeuSchWin}, we can like before associate to this an \'etale line bundle $L_n$ on $\mathfrak X_{0}/p^n$. By \'etale descent, this is associated to a finite locally free $A_0/p^n$-module $M_n$ of rank $1$. Then also $M=\varprojlim M_n$ is finite locally free with $M/p^n=M_n$ by \cite[Tag 0D4B]{StacksProject}. Passing from $\mathfrak X_0$ to any Zariski-cover where $M$ is free, any generator of $M$ induces a compatible system of $f_n\in  (A_\infty/p^n)^\times$ such that $g^{\ast}f_n=c(g)f_n$. Then $f=(f_n)_n\in \varprojlim_n (A_\infty/p^n)^\times=\O(\mathfrak X_\infty)^\times$ has the desired properties.
\end{proof}
\begin{Example}\label{exmp:CHJ}
	In order to illustrate how the above criteria can be used in practice, we now sketch various new proofs that the sheaf $\omega^{\kappa}$ of overconvergent $p$-adic modular forms defined by Chojecki--Hansen--Johansson \cite[Definition~2.18]{CHJ} is an analytic line bundle: This is a sheaf on an overconvergent neighbourhood $\mathcal X(\epsilon)$ of the ordinary locus of the modular curve. By definition, it is given by a $v$-descent datum for a certain pro-\'etale map $\mathcal X_{\Gamma(p^\infty)}(\epsilon)_a\to \mathcal X(\epsilon)$. It is therefore clear from the definition that it is a $v$-line bundle. We can now employ any of the above criteria to see that $\omega^\kappa$ is already analytic  (see \cite[\S3.4]{ChrisChris-HilbertCHJ} for more details):
	\begin{enumerate}
		\item By Corollary~\ref{c:analytic-on-Zar-dense-implies-analytic}, it suffices to prove that $\omega^\kappa$ is analytic on the ordinary locus $\mathcal X(0)\subseteq \mathcal X(\epsilon)$, which is a Zariski-dense open subspace. But here the statement is essentially classical, and originally due to Katz \cite[\S4]{p-adicMSMF}: One reduces the definition to the Igusa tower, where one has a pro-\'etale formal model, and then invokes Proposition~\ref{p:pro-etale-formal-model-implies-analytic}.
		\item By Proposition~\ref{p:non-trivial-section-implies-analytic}, it suffices to show that $\omega^{\kappa}$ has a non-trivial section. Such a section is given by the Eisenstein series (one first has to check that this matches the definition).
		\item The bundle $\omega^\kappa$ can be defined for rigid analytic families of weights $\kappa$, and naturally extends to a $v$-bundle $\omega$ on $\mathcal X(\epsilon)\times \mathcal W$ where $\mathcal W$ parametrises $p$-adic weights. We can now use Corollary~\ref{c:analytic-on-products} and check analyticity on fibres: It is easy to see that $\omega$ becomes trivial over each point of $\mathcal X(\epsilon)$. It thus suffices to prove that $\omega^{\kappa}$ is analytic for a Zariski-dense subset of weights in $\mathcal W$. Such a set is given by the classical weights.
	\end{enumerate}
\end{Example}

\section{The image of the Hodge--Tate logarithm}
Now that we have constructed the left-exact sequence
\[ 0\to \Pic_{\an}(X)\to \Pic_{v}(X)\xrightarrow{\HT\log}H^0(X,\wt\Omega^1_X),\]
we would like to determine the image of $\HT\log$ in order to give a complete answer to Question~\ref{q:howfariso}. Towards this goal, we consider in this section the one-dimensional, the affinoid, and the proper case, thus completing the proof of Theorem~\ref{t:Picv-ses}.
\subsection{The case of curves}
We start with  part \ref{enum:MT1-curves} of Theorem~\ref{t:Picv-ses}:
This says that for a smooth paracompact rigid space of pure dimension $1$ over an algebraically closed field, the above sequence
is in fact also right-exact.
We note that paracompactness is quite a weak condition:
For example, by the main theorem of \cite{LiuvdPut_separatedcurves}, any separated one-dimensional rigid space is paracompact.

The reason why this condition appears is the following lemma:
\begin{Lemma}[{\cite[Corollary 2.5.10]{deJongvdPut}}]\label{l:cohomology-in-deg-d-for-paracompact-space}
	Let $X$ be a paracompact rigid space of dimension $d$. Then $H^i_{\an}(X,F)=0$ for any abelian sheaf $F$ on $X_{\an}$ and any $i>d$.
\end{Lemma}
This is used to prove  the following lemma on the Brauer group of curves, which by the 5-term exact sequence of Theorem~\ref{t:5TES-version} completes the proof of Theorem~\ref{t:Picv-ses}.\ref{enum:MT1-curves}:
\begin{Lemma}
	Let $K$ be algebraically closed.
	Let $X$ be a paracompact rigid space of dimension $1$ over $K$. Then $H^2_{\et}(X,\O^\times)=0$.
\end{Lemma}
\begin{proof}
	This is proved by Berkovich for good $k$-analytic spaces \cite[Lemma 6.1.2]{Berkovich_EtaleCohom}, which via the comparison to rigid spaces \cite[\S8.3, Theorem~8.3.5]{huber2013etale} proves the result for taut rigid spaces. More generally, one can also argue purely in the rigid analytic category:
	
	Let $r\colon X_{\et}\to X_{\an}$ be the natural morphism of sites. Then we have
	\[Rr_{\ast}\O^\times=\O^\times\]
	by \cite[Lemma~8.3.1, Proposition~8.2.3  and Corollary~8.3.2]{FvdP}.
	Thus the natural map
	\[ H^2_{\an}(X,\O^\times)\to H^2_{\et}(X,\O^\times)\]
	is an isomorphism. But the left hand side vanishes by Lemma~\ref{l:cohomology-in-deg-d-for-paracompact-space} as $X$ is paracompact.
\end{proof}
The case of curves has a few interesting consequences for the general case, which we will later use to compute $\Pic_v(\A^n)$. These are based on functoriality of $\HT\log$:

\begin{Remark}\label{r:pullback-of-differentials}
	A general strategy to describe the image of $\HT\log$ is as follows: If $f:X\to Y$ is a morphism of smooth rigid spaces, then by functoriality we obtain a commutative diagram
	\begin{equation}\label{eq:surjective-on-differentials}
		\begin{tikzcd}[column sep = 2cm]
			\Pic_v(Y) \arrow[d,"f^{\ast}"] \arrow[r," \mathrm{HT}_Y\log"] & {H^0(Y,\wt\Omega^1_Y)} \arrow[d,"f^{\ast}"] \\
			\Pic_v(X) \arrow[r," \mathrm{HT}_X\log"]           & {H^0(X,\wt\Omega^1_X)}.                   
		\end{tikzcd}
	\end{equation}
	In particular, we have $f^{\ast}(\im\mathrm{HT}_Y\log)\subseteq \im \mathrm{HT}_X\log$.
	For example, one could use this to reduce the case of projective $X$ in Theorem~\ref{t:Picv-ses}.2 to that of abelian varieties via the Albanese variety $X\to A$. But this no longer works in general for proper $X$, see \cite[Example 5.6]{HansenLi_HodgeSymmetry}.
\end{Remark}

\begin{Corollary}\label{c:df-in-image}
	Let $X$ be any smooth rigid space. Then for any 
	$f\in \O(X)$, the differential $df\in H^0(X,\Omega^1)$ is in the image of $\HT\log\{1\}:\Pic_v(X)\{1\}\to H^0(X,\Omega^1_X)$.
\end{Corollary}
\begin{proof}
	Associated to $f$ we have a map $f:X\to \A^1$ that sends the parameter $T$ on $\A^1$ to $f$.
	Since $\A^1$ is a paracompact curve,  Theorem~\ref{t:Picv-ses}.\ref{enum:MT1-curves} shows that $\Pic_v(\A^1)=H^0(\A^1,\wt\Omega^1)$.
	The desired statement now follows  from Remark~\ref{r:pullback-of-differentials} since $f^{\ast}$ sends $dT\mapsto df$.
\end{proof}

\subsection{The cokernel in the affinoid case}
Next, we prove
part 3 of Theorem~\ref{t:Picv-ses}, which is also an easy consequence of Proposition~\ref{p:exponential-on-R^1O->R^1O^x}: We need to see that for $X$ an affinoid smooth rigid space, we get a short exact sequence
	\[ 0\to \Pic_{\an}(X)\tf\to \Pic_{v}(X)\tf\to H^0_{\an}(X,\wt\Omega^1_X)\to 0.\]
\begin{proof}[Proof of Theorem~\ref{t:Picv-ses}.3]
	The morphism of Leray 5-term exact sequences associated to the exponential \eqref{eq:exp} gives a commutative diagram of connecting homomorphisms
	\[\begin{tikzcd}
		{H^0(X,R^1\nu_{\ast}\O)} \arrow[d,"\sim"labelrotate,"\exp"] \arrow[r] & {H^2_{\et}(X,\O)} \arrow[d,,"\exp"] \\
		{H^0(X,R^1\nu_{\ast}\O^\times\tf)} \arrow[r] & {H^2_{\et}(X,\O^\times\tf)},
	\end{tikzcd}\]
	where the map on the left is an isomorphism by Proposition~\ref{p:exponential-on-R^1O->R^1O^x}.
	It therefore suffices to see that the top morphism is zero, as then so is the bottom one. But since $X$ is affinoid,\[H^2_{\et}(X,\O)=H^2_{\an}(X,\O)=0,\]
	where the first equality holds by \cite[Proposition~8.2.3.2]{FvdP}.
\end{proof}

The remaining part of Theorem~\ref{t:Picv-ses} is the proper case \ref{enum:MT1-projective}, which is arguably the most interesting one. For this we need a further ingredient: The universal cover of $X$.

\subsection{The diamantine universal cover}
Let $X$ be a connected smooth proper rigid space over an algebraically closed $K$ and fix a base point $x\in X(K)$. Since $X$ is a locally Noetherian adic space, we have the \'etale fundamental group $\pi_{1}(X,x)$, a profinite group that governs the finite \'etale covers of $X$: More precisely, let $X_{\profet}=\mathrm{Pro}(X_{\mathrm{f\acute{e}t}})$ be the category of pro-finite-\'etale covers of $X$. Let $\pi_{1}(X,x)\mathrm{-pfSets}$ be the category of profinite sets with a continuous $\pi_{1}(X,x)$-action. Then:
\begin{Proposition}[{\cite[Proposition~3.5]{Scholze_p-adicHodgeForRigid}}]
		There is an equivalence of categories
		\begin{alignat*}{2}
		F:X_{\profet}&\:\to\:&& \pi_{1}(X,x)\text{-}\mathrm{pfSets}\\
		(Y_i)_{i\in I} &\:\mapsto\:&& F(X):=\varprojlim_{i\in I}\Hom_X(x,Y_i).
		\end{alignat*}
\end{Proposition}
In particular, we have a universal object in $X_{\profet}$, which corresponds to $\pi_{1}(X,x)$ endowed with the translation action on itself. Since cofiltered inverse limits exists in the category of diamonds \cite[Lemma 11.22]{etale-cohomology-of-diamonds}, we can associate a diamond to this object:
\begin{Definition}\label{d:universal-cover}
	The universal pro-finite-\'etale cover $\wt X\to X$ is defined as the diamond
	\[\wt X:=\varprojlim_{X'\to X}X'\]
	where the index category consists of all connected finite \'etale covers $(X',x')\to (X,x)$ with $x'\in X'(K)$ a choice of lift of the base point $x\in X(K)$. This is a spatial diamond, and the canonical projection
	\[\wt X\to X\]
	is a pro-finite-\'etale $\pi_{1}(X,x)$-torsor in a canonical way. Here the additional datum of the lift $x'$ in the index category is necessary to make this action canonical, and to make the association $X\mapsto \wt X$ functorial in a canonical way. It gives a distinguished point $\wt x\in \wt X(K)$.
\end{Definition}
\begin{Example}
	\begin{itemize}
	\item If $X$ is an abelian variety, or more generally an abeloid variety, then
	\[ \wt X=\textstyle\varprojlim_{[n]}X\]
	is the limit over multiplication by $n$ where $n$ ranges through $\N$. This is represented by a perfectoid space \cite[Corollary~5.9]{perfectoid-covers-Arizona} which has the interesting feature that it is ``$p$-adic locally constant in $X$'', i.e.\ many different $X$ have isomorphic $\wt X$ \cite{heuer-isoclasses}.
	
	\item If $X$ is a connected smooth proper curve of genus $\geq 1$, then $\wt X$ is also represented by a perfectoid space \cite[Corollary~5.7]{perfectoid-covers-Arizona}, and has first been considered by Hansen.
	
	\item In the other extreme, if $X$ is a space without any non-split finite \'etale covers, e.g.\ $X=\P^n$ or $X$ a K3 surface, then we simply have $\wt X=X$. In particular, $\wt X$ is not always perfectoid.
	We do not know if $X$ is always represented by an adic space.
	\end{itemize}
\end{Example}
We call $\wt X$ the universal pro-finite-\'etale cover due to the following universal property:
\begin{Lemma}\label{l:profet-UP-of-wtX}
	Let $Y\to X$ be any pro-finite-\'etale cover, i.e.\ an element of $X_{\profet}$,  and fix a lift $y\in Y(K)$ of $x$. Then there is a unique morphism $(\wt X,\wt x)\to (Y,y)$ over $X$.
\end{Lemma}
\begin{proof}
	By the limit property, it suffices to see this for finite \'etale $Y\to X$. Passing to the connected component of $y$, we see $(Y,y)$ appears in the index of the limit defining $\wt X$.
\end{proof}
More interestingly, $\wt X\to X$ is also a topological universal cover in the following sense:
\begin{Proposition}\label{p:universal-cover}
	Let $X$ be a connected smooth proper adic space over $K$. Then for any $n\in \N$ and $F$ any of the $v$-sheaves $\Z/n, \Z_p, \wh \Z, \O^{+a}/p^n, \O^{+a}, \O, U, \O^{\times,\tt}$, we have
\begin{alignat*}{3}
H^0(\wt X,F)&=&&H^0(\Spa(K),F),\\
H^1_{v}(\wt X,F)&=&&0.
\end{alignat*}
\end{Proposition}
\begin{Remark}
	If $X$ is either a curve of genus $g\geq 1$ or an abeloid variety, we in fact have $H^i_{v}(\wt X,-)=0$ for all $i\geq 1$ for all of these coefficients \cite[Theorem~3.9]{heuer-isoclasses}. But for a general smooth proper rigid space $X$, this is no longer true as the example of $\P^1$ shows.
\end{Remark}
\begin{proof}
	We start with $\Z/n$-coefficients: By  \cite[Proposition 14.9]{etale-cohomology-of-diamonds},  we have for any $i\geq 0$:
	\[H^i(\wt X,\Z/n)=\varinjlim_{X'\to X} H^i(X',\Z/n).\]
	
	For $i=0$, since each $X'$ is connected, this implies $H^0(\wt X,\Z/n)=\Z/n$. In the limit over $n\in\N$, we get $H^0(\wt X,\wh \Z)=\wh \Z$ and similarly for $\Z_p$.
	
	For $i=1$, the group $H^1_{\et}(X',\Z/n)=H^1_{v}(X',\Z/n)$ parametrises the finite \'etale $\Z/n$-torsors on $X'$. Since any $\Z/n$-torsor is trivialised by a connected finite \'etale covers of $X$, each cohomology class gets killed in the inverse system defining $\wt X$. It follows that
	\begin{equation}\label{eq:H^1(wt X,Z/p^n)}
	H^1_v(\wt X,\Z/n)=\varinjlim_{X'\to X} H^1_v(X',\Z/n)=0.
	\end{equation}
	Since the $v$-site is replete, we have $\Rlim \Z/n=\wh\Z$ as sheaves on $X_v$ by \cite[Proposition 3.1.10]{bhatt-scholze-proetale}, so the Grothendieck spectral sequence for $\mathrm{R}\Gamma(\wt X,-)\circ \Rlim$ yields an exact sequence
	\[0\to \textstyle\mathrm{R}^1\lim_nH^0(\wt X,\Z/n)\to H^1_{v}(\wt X,\wh\Z)\to \textstyle\varprojlim_n H^1_v(\wt X,\Z/n)\to 0.\]
	The first term vanishes by the first part. The last also vanishes, so $H^1_{v}(\wt X,\wh\Z)=0$ as desired.
	
	The case of $\Z_p$-cofficients follows as $\Z_p$ is a direct factor of $\wh \Z$.
	
	To see the last part, we use the Primitive Comparison Theorem~\cite[Theorem 5.1]{Scholze_p-adicHodgeForRigid}, according to which we have for any $i\geq 0$ and $m\geq 0$ and any finite \'etale cover $X'\to X$
	\[ H^i(X',\O^+/p^m)\aeq H^i(X',\Z/p^m)\otimes_{\Z_p}\O_K.\]
	
	For $i=0$, we deduce in the limit that $ H^0(\wt X,\O^+)\aeq \O_K$.
	For $i=1$, we conclude from \eqref{eq:H^1(wt X,Z/p^n)} applied to $n=p^m$ that
	\[ H^1_v(\wt X,\O^+/p^m)\aeq \varinjlim_{X'\to X}H^1_v( X',\O^+/p^m)\aeq\varinjlim_{X'\to X}H^1_v(X',\Z/p^m)\otimes_{\Z_p}\O_K=0.\]
	It then follows from the same $\Rlim$-argument as above that 
	\[ H^1_v(\wt X,\O^+)\aeq \textstyle\varprojlim_m H^1_v(\wt X,\O^+/p^m)\aeq 0.
	\]
	The case of $U$ follows from the long exact sequence of the logarithm \eqref{eq:log}. The case of $\O^{\times,\tt}$ similarly follows from a logarithm sequence modified to include all unit roots $\mu$:
	\[1\to \mu\to \O^{\times,\mathrm{tt}}\xrightarrow{\log} \O\to 0\qedhere,\]
\end{proof}

Our guiding analogy in the following will be that $\wt X\to X$ behaves like the topological universal cover in complex geometry. We are going to make this more precise in the next section, but as 
a first instance, we recover the statement (cmp.\ \cite[Theorem 1.2]{Scholze_p-adicHodgeForRigid}):
\begin{Corollary}\label{c:pi_1^p-ab-fg}
	Let $T$ be the maximal torsionfree abelian  pro-$p$-quotient of $\pi_1(X,x)$. Then $T$ is a finite free $\Z_p$-module, and there is a natural isomorphism
	\[ T=\Hom( H^1_{\et}(X,\Z_p),\Z_p).\]
\end{Corollary}
\begin{proof}
	By Proposition~\ref{p:universal-cover} and  Proposition~\ref{p:Cartan--Leray}.\ref{i:CL-left-ex} (Cartan--Leray) for $\wt X\to X$ with $\mathcal F=\Z_p$, we have $\Hom(\pi_1(X,x),\Z_p)=H^1_{\et}(X,\Z_p)$. The result follows by applying $\Hom(-,\Z_p)$. That $T$ is finite free follows from $H^1_{\et}(X,\Z_p)$ being finitely generated \cite[ Theorem 1.1]{Scholze_p-adicHodgeForRigid}.
\end{proof}
The relevance of the universal cover $\wt X$ to Theorem~\ref{t:Picv-ses}.\ref{enum:MT1-projective} is now the following:
\begin{Corollary}\label{c:CL-for-GL_n}
	For any $n\geq 1$, there is a short exact sequence of pointed sets
	\[ 1\to \Hom_{\cts}(\pi_1(X,x),K^\times)\to \Pic_v(X)\to \Pic_v(\wt X).\]
\end{Corollary}
\begin{proof}
	This follows from Corollary~\ref{c:Cartan--Leray-for-line-bundles} (Cartan--Leray) applied to the pro-finite-\'etale $\pi_1(X,x)$-torsor $\wt X\to X$ and $\mathcal F=\O^\times$, and the fact that $\O(\wt X)=K$ by Proposition~\ref{p:universal-cover}.
\end{proof}
We can thus see characters of $\pi_1(X,x)$ as descent data on the trivial line bundle on $\wt X$.
This is part of a much more general picture that we study in the next section. For now, the crucial point is that it gives us ``enough'' $v$-line bundles in $\Pic_v(X)$ to generate $H^0(X,\widetilde\Omega^1_X)$.

\subsection{The proper case}
We now have everything in place to finish the remaining case of Theorem~\ref{t:Picv-ses}:
\begin{proof}[Proof of Theorem~\ref{t:Picv-ses}.\ref{enum:MT1-projective}]
	By passing to connected components, we may without loss of generality assume that $X$ is connected. Fix a base point $x\in X(K)$.
	
	Recall from the proof of Theorem~\ref{t:Picv-ses}.1 that the term $H^0(X,\Omega^1_X)\{-1\}$ arises from the Leray spectral sequence as $H^0(X,R^1\nu_{\ast}\O^\times)$. We now compare this to the Leray spectral sequence for $\O$, which we recall gives the Hodge--Tate spectral sequence.
	By \cite[Theorem~13.3.(ii)]{BMS}, the latter degenerates at the $E_2$-page since $X$ is proper. Consequently,
	\[\HT: H^1_v(X,\O)\to H^0(X,\widetilde\Omega^1_X) \]
	is surjective.
	
	We now compare this to the Cartan--Leray  sequences of Proposition~\ref{p:Cartan--Leray}.\ref{i:CL-left-ex} for 
	$\wt X\to X$.
	By Proposition~\ref{p:universal-cover}, we have $H^1_v(\wt X,\O)=0$. Hence the Cartan--Leray sequence of $\O$ is of the form
	\[0\to  \Hom_{\cts}(\pi_1(X,x),K)\to H^1_v( X,\O)\to H^1_v(\wt X,\O)=0.\]
	Similarly, by Corollary~\ref{c:CL-for-GL_n}, there is a contribution of $ \Hom_{\cts}(\pi_1(X,x),K^\times)$ to $\Pic_v(X)$. Passing from $\O^\times$ to  $U=1+\m\O^+\subseteq \O^\times$, we
	 compare these Cartan--Leray sequences via the logarithm $\log:U\to \O$, and get by construction of $\HT\log$ a commutative diagram 
	\begin{equation}\label{eq:Rep-to-diff-surj-proper-case}
	\begin{tikzcd}
		 {\Hom_{\cts}(\pi_1(X,x),1+\m)}     \arrow[r]        \arrow[d,"\log"']  & {H^1_{v}(X,U)}    \arrow[d,"\log"'] \arrow[r]&    {H^1_{v}(X,\O^\times)}  \arrow[d,"\HT\log"]                   \\
		 {\Hom_{\cts}(\pi_1(X,x),K)}\arrow[r,"\sim"] & {H^1_v(X,\O)} \arrow[r, two heads,"\HT"] & {H^0(X,\widetilde\Omega^1)}.
	\end{tikzcd}
	\end{equation}
	To prove that $\HT\log$ is surjective,
	it thus remains to see that the left vertical map is surjective. To see this, we note that any continuous homomorphism $\varphi:\pi_1(X,x)\to K$ factors through the maximal torsionfree abelian pro-$p$-quotient, which is a finite free $\Z_p$-module by Corollary~\ref{c:pi_1^p-ab-fg}. We can thus lift it to a continuous homomorphism
	\[\pi_1(X,x)\to 1+\m\subseteq K^\times\] since $\log:1+\m\to K$ is surjective, $K$ being algebraically closed.
\end{proof}

\section{Application to the $p$-adic Simpson correspondence}
Let $K$ be an algebraically closed complete extension of $\Q_p$.
Then the proper case of Theorem~\ref{t:Picv-ses}.2  is very closely related to the $p$-adic Simpson correspondence from the pro-étale/$v$-topological
perspective of \cite[\S2]{LiuZhu_RiemannHilbert}\cite[\S3]{wuerthen_vb_on_rigid_var}\cite[\S7]{MannWerner_LocSys_p-adVB}: In this section, we show that Theorem~\ref{t:Picv-ses} can be used to construct the $p$-adic Simpson correspondence in rank 1.

\subsection{Overview}
In order to provide some context, let us briefly describe a few known results about the $p$-adic Simpson correspondence. We refer to \cite[\S1]{wuerthen_vb_on_rigid_var} for a much more detailed overview. 

Let $X$ be a connected proper smooth rigid space over $K$. Fix a base point $x\in X(K)$.
Inspired by the complex Corlette--Simpson correspondence \cite{SimpsonCorrespondence}, the $p$-adic Simpson correspondence pioneered independently by Deninger--Werner \cite{DeningerWerner_vb_p-adic_curves} and Faltings \cite{Faltings_SimpsonI} is a conjectural equivalence between the category $\mathrm{Rep}_{K}(\pi_1(X,x))$ of continuous representations 
\[\pi_1(X,x)\to \GL(W)\]
on finite dimensional $K$-vector spaces $W$,
and a certain subcategory of the Higgs bundles on $X$, yet to be identified. Here by a Higgs bundle we shall mean a pair $(E,\theta)$ of an analytic vector bundle $E$ on $X$ together with a $1$-form $\theta\in H^0(X,\End(E)\otimes \wt\Omega_X^1)$ satisying $\theta\wedge \theta=0$. Such $\theta$ are called Higgs fields. We recall that $\wt\Omega_X^1:=\Omega^1_X(-1)$ where the $(-1)$ is a Tate twist; it is natural to include it in this context since it appears in the $p$-adic Hodge--Tate sequence.

In the case that $K=\C_p$ and $X$ is algebraic and defined over a finite extension of $\Q_p$,  Deninger--Werner have identified a category $\mathcal B^s(X_{\C_p})$ of algebraic vector bundles $V$ with ``numerically flat reduction'' for which they can construct a functor  \cite[\S9-\S10]{DeningerWerner_Simpson} \[\mathcal B^s(X_{\C_p})\to \mathrm{Rep}_{\C_p}(\pi_1(X,x)),\]
generalising their earlier work in the case of curves \cite[Theorem~1.1]{DeningerWerner_vb_p-adic_curves}. This gives the desired functor in the case of vanishing Higgs field, i.e.\ $\theta=0$.

W\"urthen has recently extended this to the setting of proper connected seminormal rigid analytic varieties over $\C_p$, for which he constructs a fully faithful functor on analytic vector bundles $E$ \cite[Theorem 1.1]{wuerthen_vb_on_rigid_var}. Moreover, he shows that the condition of numerically flat reduction implies that $E$ is trivialised by a pro-finite-\'etale cover of $X$ \cite[Proposition 4.13]{wuerthen_vb_on_rigid_var}. Passing from the analytic to the $v$-topology, Mann--Werner \cite[Theorem 0.1]{MannWerner_LocSys_p-adVB} extend this to $v$-vector bundles, and show that the condition of numerically flat reduction can be checked on proper covers. They then set up a category equivalence of such $v$-vector bundles to $\C_p$-local systems on $X$ that admit an integral model over $\O_{\C_p}$.

In an independent line of research, for algebraic $X$ that have an integral model with toroidal singularities over a complete discrete valuation ring, Faltings constructed an equivalence of categories from ``small'' Higgs bundles to a category of ``small generalised representations'' \cite[Theorem~5]{Faltings_SimpsonI}. Here generalised representations form a category into which representations of $\pi_1(X,x)$ embed fully faithfully. He then proved that the smallness assumption can be removed for curves  \cite[Theorem~6]{Faltings_SimpsonI}. This construction was further developed by Abbes--Gros and Tsuji \cite{AGT-p-adic-Simpson}. However, towards a $p$-adic Simpson correspondence, it is currently not known which Higgs bundles correspond to actual representations.

Reinterpreting these objects in the setting of Scholze's $p$-adic Hodge theory, Liu--Zhu were able to define a functor from $\Q_p$-local systems on any smooth rigid space defined over a finite extension of $\Q_p$ to nilpotent Higgs bundles \cite[Theorem~2.1, Remark~2.6]{LiuZhu_RiemannHilbert}. But it is not clear how this can be extended to a functor on all of $\mathrm{Rep}_{K}(\pi_1(X,x))$.

Despite these many recent advances, a construction of a more general functor either from Higgs bundles beyond the case of $\theta=0$, or from all $K$-linear representations beyond small or $\Q_p$-representations has not been found yet.
\subsection{Pro-finite-\'etale vector bundles via the universal cover}
The aim of this section is to construct the $p$-adic Simpson correspondence of rank $1$ in full generality, i.e.\ for smooth proper rigid spaces defined over any algebraically closed non-archimedean extension $K$ of $\Q_p$. Here we note that in rank $1$, a Higgs bundle is simply a pair $(L, \theta)$ of an analytic line bundle $L$ on $X$ and a global differential $\theta \in H^0(X,\widetilde\Omega^1)$, which is automatically a Higgs field. The basic idea is that by Theorem~\ref{t:Picv-ses}.\ref{enum:MT1-projective}, Higgs bundles of rank $1$ are essentially the $v$-line bundles, at least after certain choices. Under this correspondence, the condition of vanishing Chern classes in the complex case is replaced by the following:

\begin{Definition}
	\begin{enumerate}
	\item We say that a $v$-vector bundle on $X$ is pro-finite-\'etale if it is trivialised by a pro-finite-\'etale cover of $X$. Equivalently, it is trivialised by the universal cover $\wt X\to X$ from Definition~\ref{d:universal-cover}. 
	We denote by $\Pic_{\profet}(X)\subseteq \Pic_v(X)$ the subgroup of pro-finite-\'etale line bundles, and by $\Pic_{\profet,\an}(X)$ its intersection with $\Pic_{\an}(X)$.
	\item We call a Higgs bundle $(E,\theta)$ pro-finite-\'etale  if $E$ is pro-finite-\'etale.
\end{enumerate}
\end{Definition}

The first step in the complex Simpson correspondence is to associate to any finite dimensional complex representation of the fundamental group of a compact K\"ahler manifold a holomorphic vector bundle that becomes trivial on the topological universal cover. Using the $p$-adic universal cover $q:\wt X\to X$ of Definition~\ref{d:universal-cover}, we get an analogous construction:
\begin{Theorem}\label{t:CL-for-GL_n}
	Let $X$ be a connected smooth proper rigid space over $K$. Fix $x\in X(K)$. Then the universal cover $\wt X\to X$ induces an exact equivalence of tensor categories
	\begin{alignat*}{4}
		\Bigg\{\begin{array}{@{}c@{}l}\text{ finite dim.\ continuous $K$-linear}\\\text{representations of $\pi_1(X,x)$} \end{array}\Bigg\} &\isomarrow&& \Bigg\{\begin{array}{@{}c@{}l}\text{pro-finite-\'etale}\\\text{$v$-vector bundles on $X$}\end{array}\Bigg\}\\V(\wt X)\quad &\mapsfrom&&\quad V\\
		\rho:\pi_1(X,x)\to \GL(W)\quad&\mapsto &&\quad V_{\rho}
	\end{alignat*}
	where the $v$-vector bundle $V_{\rho}$ on $X$  associated to $\rho$ is defined on $Y\in X_v$ by
	\[ V_{\rho}(Y)=\{ x\in W\otimes_K\O(Y\times_X\wt X)\mid g^{\ast} x=\rho(g)x \text{ for all }g\in \pi_1(X,x)\}.\]
\end{Theorem}
\begin{proof}
	By Lemma~\ref{l:profet-UP-of-wtX}, the right hand side are precisely the $v$-vector bundles trivialised by the $v$-cover $\wt X\to X$. By Lemma~\ref{l:effective-descent}, these correspond to descent data on trivial vector bundles on $\wt X$. By Proposition~\ref{p:universal-cover}, trivial vector bundles on $\wt X$ are equivalent to finite dimensional $K$-vector spaces via the functor $W\mapsto W\otimes_K\O_{\wt X}$. The desired equivalence now follows from Corollary~\ref{c:Cartan--Leray-for-line-bundles} which implies that descent data on $W\otimes_K\O_{\wt X}$ are equivalent to continuous representations $\rho:\pi_1(X,x)\to \GL(W)$ by sending $\rho$  to $V_{\rho}$.
	
	To see that $V\mapsto V(\wt X)$ defines a quasi-inverse, observe that via $\wt X\times_X\wt X=\pi_1(X,x)\times \wt X$,
	\[ V_{\rho}(\wt X)=\{ x\in \Map_{\cts}(\pi_1(X,x),W)\mid x(g-)=\rho(g)x \text{ for all }g\in \pi_1(X,x)\}.\]
	Via the evaluation at $0$, this is in natural bijection with $W$, as we wanted to see.
	
	It is clear that both functors are exact and preserve tensors: Indeed, whether a sequence on $X$ is exact can be checked on $\wt X$, where it is exact if and only if it is on global sections.
\end{proof}
\begin{Remark}
	The same argument for $\GL_n(\O)$ replaced by $\GL_n(\O^+)$ shows that $v$-locally free $\O^+$-modules can be interpreted as the ``generalised representations'' of rank $n$ in the sense of \cite[\S2]{Faltings_SimpsonI}; This has also been observed by Liu--Zhu \cite[Remark~2.6]{LiuZhu_RiemannHilbert}.
\end{Remark}
We now apply this to formulate a $p$-adic Simpson correspondence in rank one. For this it is desirable to characterise pro-finite-\'etale line bundles on $X$ in a more explicit way:	
	\begin{Definition}
		We say that a $v$-line bundle $L$ on $X$ is topologically torsion if $L$ is in the image of
		\[H_v^1(X, \O^{\times,\mathrm{tt}}) \to H_v^1(X, \O^\times),\]
		where  $\O^{\times,\mathrm{tt}}\subseteq \O^\times$ is the topological torsion subsheaf from Definition~\ref{d:top-torsion-subsheaf}.
		We denote the subgroup of topological torsion line bundles by
		$\Pic_v^{\mathrm{tt}}(X)$.
		We denote by $\Pic_{\an}^{\mathrm{tt}}(X)$ the intersection of $\Pic_v^{\mathrm{tt}}(X)$ with $\Pic_{\an}(X)$.
	\end{Definition}
	\begin{Example}
		We will show in \cite[\S4]{heuer-diamantine-Picard} that $\Pic_{\an}^{\mathrm{tt}}(X)$ is precisely the topological torsion subgroup of $\Pic_{\an}(X)$ endowed with its natural topology as $K$-points of the rigid analytic Picard variety. If $X$ is projective with torsionfree N\'eron--Severi group and $K=\C_p$, this happens to equal $\Pic_{\an}^0(X)$ but this is no longer true for any non-trivial extension of $K$. 
		
		For example, if $X$ is an abelian variety with good reduction $\overline{X}$ over $k$, let $X^\vee$ be the dual abelian variety with its reduction $\overline{X}^\vee$. Then $\Pic_{\an}^{\mathrm{tt}}(X)$ is precisely the preimage of the torsion subgroup of $\overline{X}^\vee(k)$ under the specialisation map $\Pic^0(X)=X^\vee(K)\to \overline{X}^\vee(k)$.
	\end{Example}

\subsection{The $p$-adic Simpson correspondence for line bundles}
We can now give our second main application of Theorem~\ref{t:Picv-ses}:
\begin{Theorem}[$p$-adic Simpson correspondence of rank one]\label{t:p-adic-Simpson-rk-1}
	Let $X$ be a connected smooth proper rigid space over $K$. Fix a base point $x\in X(K)$.
	
	 \begin{enumerate}
	 	\item There is a short exact sequence, functorial in $X$,
	\[ 0\to \Pic_{\profet,\an}(X)\to \Hom_{\cts}(\pi_1(X,x),K^\times)\to H^0(X,\widetilde\Omega_X^1)\to 0.\]
	\item Any choice of a splitting of $\log\colon1+\m\to K$ as well as a splitting of the Hodge--Tate sequence define an equivalence of tensor categories
	\[\Bigg\{\begin{array}{@{}c@{}l}\text{$1$-dim.\ continuous $K$-linear}\\\text{representations of $\pi_1(X,x)$} \end{array}\Bigg\} \isomarrow \Bigg\{\begin{array}{@{}c@{}l}\text{pro-finite-\'etale analytic}\\\text{Higgs bundles on $X$ of rank }1\end{array}\Bigg\}.\]
	\item We have $\Pic_{\profet,\an}(X)=\Pic_{\an}^{\tt}(X)$, so the right hand side can equivalently be described as the topological torsion Higgs bundles.
	\end{enumerate}
\end{Theorem}
In particular, this singles out pro-finite-\'etale Higgs line bundles as the desired subcategory for the Simpson correspondence in rank $1$. Before we discuss the proof, we make some remarks on how this relates to the works discussed in the last subsection:

\begin{Remark}
	The choices made in Theorem~\ref{t:p-adic-Simpson-rk-1}.2 are essentially the same as the ones made by Faltings in his construction of a $p$-adic Simpson correspondence for small generalised representations: The only difference is that in the generality we work in, one needs to choose a splitting of the Hodge--Tate sequence. There is a canonical such splitting if $X$ is defined over a discretely valued non-archimedean extension of $\Q_p$ \cite[Corollary~1.8]{Scholze_p-adicHodgeForRigid}, which is part of the assumption of Faltings' setup. In our setup, since $X$ is quasi-compact, a choice of splitting is induced by a choice of lifting of $X$ to $B_{\dR}^+(K)/\xi^2$, which also appears in Faltings' work. This lift is arguably a ``better'' choice than that of a splitting of the map $\HT$, as the equivalence then becomes functorial in rigid spaces with a choice of lift.
\end{Remark}
\begin{Remark}	
	We note that the ``topological torsion'' condition is strictly weaker than the ``smallness'' condition imposed by Faltings in \cite[\S2]{Faltings_SimpsonI}.
\end{Remark}
\begin{Remark}
	For an analytic line bundle $L$ on $X$, the condition $L\in \Pic_{\profet,\an}(X)$ means precisely that $L$ is in the category $\mathcal B^{\mathrm{p\acute{e}t}}(\O_X)$ of ``trivialisable'' analytic vector bundles in the sense of \cite[Theorem~3.10]{wuerthen_vb_on_rigid_var}. But we explicitly also include the case of general $\theta$.
\end{Remark}
\begin{Remark}
If $X$ is algebraic, $L$ is analytic, and $K=\C_p$, then one can show that the condition from part 3 is equivalent to $L$ having numerically flat reduction in the sense of Deninger--Werner, using \cite[Proposition 4.13]{wuerthen_vb_on_rigid_var}. In this light, Theorem~\ref{t:p-adic-Simpson-rk-1} confirms at least in rank $1$ that this  is the correct replacement for the complex condition of being ``semistable with vanishing Chern classes'', also beyond the case of vanishing Higgs fields.
\end{Remark}

\begin{Remark}
	More generally, Theorem~\ref{t:CL-for-GL_n} suggests that pro-finite-\'etale Higgs bundles are a promising step towards the correct subcategory for the $p$-adic Simpson correspondence. In particular, this would mean that the functor constructed in \cite{MannWerner_LocSys_p-adVB} is already the correct functor from Higgs bundles to local systems. We will pursue this further in future work.
\end{Remark}

\begin{proof}[Proof of Theorem~\ref{t:p-adic-Simpson-rk-1}]
	The first part follows from Theorem~\ref{t:Picv-ses}.\ref{enum:MT1-projective} and Corollary~\ref{c:CL-for-GL_n}: We only need to see that the composition
	\[ \Hom(\pi_1(X,x),K^\times)\isomarrow\Pic_{\profet}(X)\subseteq \Pic_v(X)\to H^0(X,\wt\Omega^1)\]
	is surjective. But this follows from diagram~\eqref{eq:Rep-to-diff-surj-proper-case} in the proof of Theorem~\ref{t:Picv-ses}.\ref{enum:MT1-projective}.
	
	To deduce the second part, we first note that the choices made induce a splitting $s$ of
	\[ \Hom(\pi_1(X,x),1+\m)\xrightarrow{\log}\Hom(\pi_1(X,x),K)=H^1_{\an}(X,\O)\xrightarrow{\HT}H^0(X,\wt\Omega^1).\]
	In particular, they define a splitting of the sequence in part 1. This gives a bijection between isomorphism classes.
	In order to upgrade this to an equivalence of categories, we use the equivalence of Theorem~\ref{t:CL-for-GL_n}: This shows that it suffices to construct a tensor equivalence 
	\[\Bigg\{\begin{array}{@{}c@{}l}\text{pro-finite-\'etale analytic}\\\text{Higgs bundles on $X$ of rank }1\end{array}\Bigg\} \isomarrow \Bigg\{\begin{array}{@{}c@{}l}\text{pro-finite-\'etale}\\\text{$v$-line bundles on $X$}\end{array}\Bigg\}.\]
	To define this, we first observe that to any Higgs bundle $(E,\theta)$, the splitting $s$ associates a character $s(\theta):\pi_1(X,x)\to K^\times$ to which Theorem~\ref{t:CL-for-GL_n} attaches a $v$-line bundle $L_{\theta}$. We now define the functor by sending
	\[ (E,\theta)\mapsto E\otimes L_{\theta}.\]
	This is indeed functorial as any morphism of Higgs line bundles $(E_1,\theta_1)\to (E_2,\theta_2)$ is trivial unless $\theta_1=\theta_2$, in which case it is simply the datum of a morphism $E_1\to E_2$.
	
	A quasi-inverse can be defined as follows: Let $E$ be any pro-finite-\'etale $v$-line bundle on $X$, and let $\theta(E)\in H^0(X,\wt \Omega^1_X)$ be the image of the isomorphism class of $E$ under $\HT\log$. Then we define a functor by
	\[E\mapsto (E\otimes L_{\theta(E)}^{-1},\theta(E)),\]
	where the line bundle is analytic by left-exactness of the sequence in Theorem~\ref{t:Picv-ses}. 
	
	This is also functorial, for trivial reasons: For any two pro-finite-\'etale line bundles $L,L'$, the line bundle of endomorphisms $L'\otimes L^{-1}$ pulls back to the trivial bundle along $\wt X\to X$, because $L'$ and $L$ do. Since $\O(\wt X)=K$, it follows that 
	\[H^0(X,L'\otimes L^{-1})=H^0(\wt X,L'\otimes L^{-1})^{\pi_1(X,x)}\cong \begin{cases}
	K&\text{if } L'\cong L,\\0 & \text{otherwise.}
	\end{cases}\]
	We have thus constructed the desired equivalence of categories. That this is  a tensor equivalences follows from the linearity of the section $s$, which implies that
	\[ L_{\theta_1+\theta_2}=L_{\theta_1}\otimes L_{\theta_2}.\]
	
	It remains to prove part 3. This is achieved by the following lemma.
\end{proof}
\begin{Lemma}
	Inside  $\Pic_v(X)$, we have $\Pic_v^{\mathrm{tt}}(X)=\Pic_{\profet}(X)$. 
 \end{Lemma}
\begin{proof}
	For the inclusion $\supseteq$, we use that any continuous homomorphism $\pi_1(X,x)\to K^\times$ factors through $\O^{\times,\tt}(K)$. Comparing Cartan--Leray sequences, we get a commutative diagram
		\[
	\begin{tikzcd}
	0\arrow[r]&	{\Hom_{\cts}(\pi_1(X,x),\O^{\times,\tt}(K))} \arrow[r]\arrow[d,equal]        & {H^1_v( X,{\O^{\times,\tt})}}  \arrow[d] \arrow[r] &H^1_v(\wt X,\O^{\times,\mathrm{tt}})\arrow[d]\\
			0\arrow[r]&{\Hom_{\cts}(\pi_1(X,x),K^\times)} \arrow[r] & {H^1_v(X,\O^\times)}\arrow[r] &H^1_v(\wt X,\O^{\times}).
	\end{tikzcd}\]
	By Corollary~\ref{c:CL-for-GL_n}, the image of the bottom left map is precisely $\Pic_{\profet}(X)$. The diagram shows that this is included in the image of the map in the middle,  which is precisely $\Pic_v^{\mathrm{tt}}(X)$.

	By the same diagram, the inclusion $\subseteq$ holds as $H^1_v(\wt X,\O^{\times,\mathrm{tt}})=1$ by Proposition~\ref{p:universal-cover}.
\end{proof}

\section{The $v$-Picard group of affine space $\A^n$}
By the Theorem of Quillen--Suslin, every vector bundle on $\Spec(K[X_1,\dots,X_n])$ is trivial. Similarly, every analytic vector bundle on the rigid affine space $\A^n$ is trivial \cite[\S V.3 Proposition 2.(ii)]{Gruson_FibresVect}. In this final section, we prove that this is no longer true in the $v$-topology:

\begin{Theorem}\label{t:Pic-v-A^n}
	For any $n\in \N$, the Hodge--Tate logarithm induces an isomorphism
	\[ \Pic_v(\A^n)=H^0(\A^n,\wt\Omega^1)^{d=0}.\]
	More generally, for $k\geq 0$ and $n\geq 1$, we have $\Pic_v(\G_m^k\times \A^n)=H^0(\G_m^k\times\A^n,\wt\Omega^1)^{d=0}$.
\end{Theorem}

\begin{Remark}
	We are interested in the case of $\G_m\times \A$ since its de Rham complex is no longer exact on global sections. This shows that we should really think of line bundles as living in $(\wt\Omega^1)^{d=0}$ rather than $d(\O)$. This is also evidenced by the $1$-dimensional case.
\end{Remark}
\begin{Definition}
We denote by $B^n=\Spa(K\langle T_1,\dots,T_n\rangle)\subseteq \A^n$ the closed unit ball. For any $s\in |K|$, denote by $B_s^n\subseteq \A^n$ the closed ball defined by $|T_i|\leq s$ for all $i=1,\dots,n$.
\end{Definition}
We often use that $B_s^n\cong B^n$ by rescaling. In particular, $\Pic_{\an}(B_s^n)=1$ by \cite[Satz~1]{Lutkebohmert_Vektorraumbundel}.
\begin{Corollary}\label{c:Cor-Pic_v(B^n)}
	In contrast to Theorem~\ref{t:Picv-ses}.2, the map
	$\HT\log:\Pic_v(B^n)\to H^0(B^n,\wt\Omega^1)$
	is no longer surjective for $n\geq 2$.
\end{Corollary}
\begin{proof}
	If it was, it would be an isomorphism. Covering $\A^n$ by the $B^n_s$, this shows
	\[\Pic_v(\A^n)= \textstyle\varprojlim_s \Pic_v(B^n_s)=\textstyle\varprojlim_s H^0(B^n_s,\wt\Omega^1)=H^0(\A^n,\wt\Omega^1)\]
	where the first isomorphism uses $\Pic_{\an}(\A^n)=1$. This is
	a contradiction to Theorem~\ref{t:Pic-v-A^n}.
\end{proof}

We will give two different proofs of the first part of Theorem~\ref{t:Pic-v-A^n}: The first relies on known results about $\mathrm{R}\Gamma_{\proet}(\A^n,\Q_p)$ and the Poincar\'e Lemma in $X_{\proet}$. We note that in general, this restricts the setup to $K$ being the completion of an algebraic closure of a discretely valued field.
 The second approach is closer in spirit to classical rigid analytic computations. It is slightly more general as it does not require the Poincar\'e Lemma.
 
 \subsection{Preparations}
 We begin with some technical preparations that we will need for both proofs:
 	\begin{Lemma}\label{l:bOx-on-B^n}
 		 Let $n\geq 1$. Let $Y$ be any diamond over $\Spa(K)$. Then
 		 \begin{alignat*}{3}
 		 H^0(Y\times B^n,\bOx)&=&&H^0(Y,\bOx)\\
 		 H^1_v(B^n,\bOx)&=&&1.
 		 \end{alignat*}
	\end{Lemma}
	\begin{proof}
		For part 1, we can by induction assume that $n=1$. Since $\bOx$ is a $v$-sheaf, it suffices to prove the statement for $Y=\Spa(R,R^+)$ an affinoid perfectoid space. We then have 
		\[\O(Y\times B)=R\langle T\rangle\]
		and we need to prove that for any $f\in R\langle T\rangle^\times$ of the form $f=1+a_1T+a_2T^2+\dots$, we have $a_i\in \m R^+=R^{\circ\circ}$ for all $i\geq 1$. But this we can check on points of $Y$, which reduces us to the case of $(R,R^+)=(C,C^+)$ a field. Since $\m C^+=\m\O_C$ does not depend on $C^{+}$, this is a statement about rigid geometry: Here, by \cite[\S5.3.1 Proposition 1]{BGR} we indeed have
		\[ C\langle R\rangle^{\times}=C^\times\times (1+\m\O_C\langle T\rangle).\]
		For part 2, Lemma~\ref{l:R^1nu_*bOx}.2 reduces to showing $H^1_{\et}(B^n,\bOx)=1$. This follows from the exponential sequence \eqref{eq:exp} as $\Pic_{\et}(B^n)=1$  and $H^2_{\et}(B^n,\O)=0$.
	\end{proof}
\begin{Lemma}\label{l:H(A^n,U)=H(A^n,Ox)}
	For any $n\geq 1$, we have
	$H^1_v(\A^n,U)= H^1_v(\A^n,\O^\times)$.
\end{Lemma}
\begin{proof}
	It suffices to prove that the second and fourth map in the long exact sequence
	\[ H^0(\A^n,\O^\times)\to H^0(\A^n,\bOx)\to H^1_v(\A^n,U)\to H^1_v(\A^n,\O^\times)\to H^1_v(\A^n,\bOx) \]
	are trivial. This follows from Lemma~\ref{l:bOx-on-B^n} by the \cH-to-sheaf sequence of  $\A^n=\cup_{s\in\N} B_s^n$.
\end{proof}
 
\subsection{Proof via comparison to pro-\'etale cohomology}
In this section, let us assume that $K=\C_p$. Then Theorem~\ref{t:Pic-v-A^n} is closely related to a result of
	Colmez--Nizio{\l} \cite{ColmezNiziol_CohomologyAffine}, and independently of Le Bras \cite{LeBras-Espaces}, who both show:
	\begin{Theorem}[{\cite[Theorem 1]{ColmezNiziol_CohomologyAffine}, \cite[Th\'eor\`eme 3.2]{LeBras-Espaces}}]\label{t:Colmez-Niziol-LeBras}
	Over $\C_p$, we have for all $i\geq 1$:
	\begin{equation}\label{eq:LeBras} H^i_{\proet}(\A^n,\Q_p)=H^0(\A^n,\wt \Omega^{i})^{d=0}.
	\end{equation}
	\end{Theorem}
	
	In this subsection, we explain how Le Bras' proof of this result can be used to prove Theorem~\ref{t:Pic-v-A^n} over $\C_p$.
	For this, we first note that $H^i_{\proet}(\A^n,\Z/p^n)=H^i_{\et}(\A^n,\Z/p^n)=0$ for $i\geq1$, and thus
	\[\mathrm{R}\Gamma_{\proet}(\A^n,\Z_p)=\Rlim \mathrm{R}\Gamma_{\proet}(\A^n,\Z/p^n)=\Z_p,\]
 	which shows
	\begin{equation}\label{eq:RG_proet(A^n,Q_p vs mu_p)} H^i_{\proet}(\A^n,\Q_p(1))=H^i_{\proet}(\A^n,\mu_{p^\infty}).
	\end{equation}
	
	\begin{Proposition}\label{p:Pic(A^n)-via-H^1(A^n,Q_p)}
	The long exact sequence of the logarithm \eqref{eq:log} induces a sequence
	\[0\to H^1_{\proet}(\A^n, \O^\times)\xrightarrow{\log} H^1_{\proet}(\A^n,\O)\to H^2_{\proet}(\A^n,\mu_{p^\infty})\to 0 \]
	which is short exact and isomorphic to the $(-1)$-twist of the sequence
	\[0\to H^0(\A^n,\Omega^1)^{d=0}\to  H^0(\A^n,\Omega^1)\xrightarrow{d}H^0(\A^n,\Omega^{2})^{d=0}\to 0.\]
	\end{Proposition}
	\begin{proof}
		Let $X=\A_{\Q_p}^n$ be the rigid affine space over $\Q_p$, so that $\A^n=\A^n_{\C_p}=X_{\C_p}$. Choose an isomorphism $\Z_p\cong \Z_p(1)$, i.e.\ a compatible system of $p^n$-th unit roots $\epsilon\in \C_p^\flat=\varprojlim_{x\mapsto x^p}\C_p$.
		
		In order to prove Theorem~\ref{t:Colmez-Niziol-LeBras}, Le Bras considers Colmez' fundamental exact sequence \cite[Proposition 8.25.3]{Colmez-espacesdeBanach} in its incarnation in terms of period sheaves on $X_{\proet}$:
		\begin{equation}\label{eq:fundamental-sequence-of-p-adic-Hodge-theory}
		0\to \Q_p\to \B[\tfrac{1}{t}]^{\varphi=1}\to \BdR/\BdRp\to 0
		\end{equation}
		where $t=\log([\varepsilon])$ (see \cite[\S8]{LeBras-Espaces} for the definition of these sheaves).
		For $i>0$, he shows $H^i_{\proet}(\A^n,\B[\tfrac{1}{t}]^{\varphi=1})=0$ and $H_{\proet}^i(\A^n,\BdR)=0$. This gives an isomorphism for $i>1$:
		\[ H^i_{\proet}(\A^n,\Q_p)\isomarrow H^i_{\proet}(\A^n,\BdRp).\]

		As pointed out to us by Le Bras, this isomorphism is related to our setting in \S2 by way of the following commutative diagram of sheaves on $X_{\proet}$ with short exact rows:
		\[
		\begin{tikzcd}[row sep = 0.4cm]
			0 \arrow[r] & \mu_{p^\infty} \arrow[r]           & U \arrow[r,"\log"]                                                  & \O \arrow[r]                                                   & 0 \\
			0 \arrow[r] & \Q_p(1) \arrow[u] \arrow[r] \arrow[d,"\sim"labelrotate] & \B[\tfrac{1}{t}]^{\varphi=p} \arrow[r] \arrow[u] \arrow[d, "\cdot t^{-1}"] & \BdRp/{t\BdRp} \arrow[u,"\sim"labelrotatep,"\theta"] \arrow[r] \arrow[d, "\cdot t^{-1}"] & 0 \\
			0 \arrow[r] & \Q_p \arrow[r]                     & \B[\tfrac{1}{t}]^{\varphi=1} \arrow[r]                                     & \BdR/\BdRp \arrow[r]                                        & 0.
		\end{tikzcd}\]

	Using  \eqref{eq:RG_proet(A^n,Q_p vs mu_p)}, the $5$-Lemma and Lemma~\ref{l:H(A^n,U)=H(A^n,Ox)}, the top two lines induce an isomorphism
	\[\Pic_{v}(\A^n)=H^1_{\proet}(\A^n,U)\isomarrow H^1_{\proet}(\A^n,\B^{\varphi=p}).\]
	From the bottom two rows, we thus get a morphism of long exact sequences
		\[
	\begin{tikzcd}
			\dots \arrow[r]&\Pic_{v}(\A^n)\arrow[r,"\log"] \arrow[d]&{H^1_{\proet}(\A^n,\BdRp/t\BdRp)} \arrow[r] \arrow[d,"\cdot t^{-1}"] & {H^{2}_{\proet}(\A^n,\Q_p(1))} \arrow[d,"\sim"labelrotate] \arrow[r]&\dots \\
		\dots \arrow[r]&	0\arrow[r]&{H^1_{\proet}(\A^n,\BdR/\BdRp)} \arrow[r,"\sim"]&      {H^{2}_{\proet}(\A^n,\Q_p)}\arrow[r]&\dots 
	\end{tikzcd}\]
	The map on the top left is injective: This follows from  Theorem~\ref{t:Picv-ses}.1 using  $\Pic_{\et}(\A^n)=1$.
	Consequently, using that the bottom map is an isomorphism, $\Pic_v(\A^n)$ gets identified with the kernel of the middle vertical map. Using
	$H^1_{\proet}(\A^n,\BdR)=0$ and the diagram
	\[\begin{tikzcd}
		0 \arrow[r] & t\BdRp \arrow[r] \arrow[d, "\sim"labelrotate,"\cdot t^{-1}"] & \BdRp \arrow[d, "\cdot t^{-1}"] \arrow[r] & \BdRp/t\BdRp \arrow[d, "\cdot t^{-1}"] \arrow[r] & 0 \\
		0 \arrow[r] & \BdRp \arrow[r] & \BdR \arrow[r] & \BdR/\BdRp \arrow[r] & 0,
	\end{tikzcd}\]
	this can in turn be identified with the kernel of the boundary map
	\begin{equation}\label{eq:boundary-map-comparison-to-LeBras}
	H^1_{\proet}(\A^n,\BdRp/t\BdRp)\to H^2_{\proet}(\A^n,t\BdRp).
	\end{equation}
	This can now be understood via Scholze's Poincar\'e Lemma \cite[Corollary 6.13]{Scholze_p-adicHodgeForRigid} and its Corollaries \cite[Remarque~3.18]{LeBras-Espaces}\cite[Corollary~3.2.4]{Bosco_p-adicCohomology}: For $\nu:\A^n_{\proet}\to \A^n_{\et}$, we have
	\[ R\nu_{\ast}\BdRp = \Big(\O_{X}\hotimes_{\Q_p}B_\dR^+ \xrightarrow{d}  \Omega^1_{X}\hotimes_{\Q_p} t^{-1}B_\dR^+ \to\dots\Big).\]
	Here following \cite[before Proposition~3.16]{LeBras-Espaces}, for any vector bundle $F$ on $X_{\et}$, the sheaf $F\hotimes_{\Q_p}B_\dR^+$ is defined, via the equivalence $X_{\C_p,\et}=\varprojlim_{L|\Q_p} X_{L,\et}$ where $L|\Q_p$ ranges over all finite extensions, as the compatible system of sheaves $F_{X_L}\hotimes_{L}B_\dR^+$.
	
	It follows from this that we get a distinguished triangle in $D(\A^n_{\et})$, written vertically
	\[
	\begin{tikzcd}[row sep = 0.3cm]
		R\nu_{\ast}t\BdRp \arrow[d] \arrow[r,"\sim"] & \Big(\O_X\hotimes_{\Q_p}tB_\dR^+ \arrow[r, "d"] \arrow[d]    & \Omega^1_X\hotimes_{\Q_p} B_\dR^+ \arrow[d] \arrow[r, "d"]            & \dots\Big) \\
		R\nu_{\ast}\BdRp \arrow[d] \arrow[r,"\sim"]  & \Big(\O_X\hotimes_{\Q_p}B_{\dR}^+ \arrow[r, "d"]\arrow[ur,dotted] \arrow[d] & \Omega^1_{X}\hotimes_{\Q_p}t^{-1}B_{\dR}^+ \arrow[d] \arrow[r, "d"]\arrow[ur,dotted]  & \dots\Big) \\
		R\nu_{\ast}\BdRp/t\BdRp \arrow[r,"\sim"]    & \Big(\O_X \arrow[r, "0"]                              & \wt\Omega^1 \arrow[r, "0"]                                & \dots\Big),
	\end{tikzcd}\]
	where the right hand side is in fact a short exact sequence of complexes. Chasing the diagram, this shows that the kernel of \eqref{eq:boundary-map-comparison-to-LeBras} gets identified with that of the $(-1)$-twist of
	\[ d:H^0(\A^n,\Omega^1)\to H^0(\A^n,\Omega^{2})^{d=0}.\qedhere\]
	\end{proof}
	\begin{Remark}
	 Bosco \cite{Bosco_p-adicCohomology} has generalised Le Bras' method to show that more generally, for any smooth Stein space $X$ defined over a discretely valued field extension $L|\Q_p$, there is over the completion $K$ of the algebraic closure of $L$ an exact sequence
	\[ 0\to H^i_{\dR}(X_K)\otimes_Kt^{-i+1}B_{\dR}^+\to H^i(X_{K},\B_{\dR}^+)\to \wt\Omega^1(X_{K})^{d=0}\to 0.\]
	Elaborating on the above proof, one might therefore be able to show in this generality that the sequence from Theorem~\ref{t:Picv-ses}.1 is a short exact sequence
	\[0\to \Pic_{\an}(X_{K})\to \Pic_v(X_{K})\to \wt\Omega^1(X_{K})^{d=0}\to 0.\]
	\end{Remark}

\subsection{The intermediate space $\wt B\times \A$}
While being conceptually satisfying, the approach of the last section only works for $K$ a completed algebraic closure of a discretely valued field. In the rest of this section, we shall give an alternative proof of Theorem~\ref{t:Pic-v-A^n} that avoids the input of the Poincar\'e Lemma, and works over general algebraically closed complete $K$. It uses more classical methods and arguably gives a more concrete reason why $\A^n$ has a ``minimal amount'' of $v$-line bundles (``minimal'' as we know $\Pic_v(\A^n)$ must include $H^0(\A^n,\wt\Omega^1)^{d=0}$ by Corollary~\ref{c:df-in-image}).

The basic idea behind the proof is that for any rigid space $X$, the space $X\times\A$ has no more invertible global sections than $X$, and therefore has few descent data for line bundles from pro-\'etale covers. We would like to apply this to $X=B$, which has an explicit perfectoid $\Z_p$-Galois cover $\wt B$ that is easy to work with. Since $\O^\times(B\times  \A)=\O^\times(B)$, the Cartan--Leray exact sequence Corollary~\ref{c:Cartan--Leray-for-line-bundles} for the Galois cover $\wt B\times \A\to B\times \A$ is then of the form
\[ 0\to \Pic_v(B)\to \Pic_v(B\times \A)\to \Pic_v(\wt B\times \A)^{\Z_p}.\]
Using Theorem~\ref{t:Picv-ses}.\ref{enum:MT1-curves}, we would like to see that the map $\HT\log$ identifies this with
\[ 0\to H^0(B,\wt\Omega^1)\to H^0(B\times \A,\wt\Omega^1)^{d=0}\to \O(B)\hotimes_K H^0( \A,\wt\Omega^1)\to 0\]
where if $T_1$ is the coordinate on $B$ and $T_2$ that on $\A$, the last map sends $fdT_1+gdT_2\mapsto gdT_2$.
However, this fails to be exact because the de Rham complex of $B$ is not exact on global sections. As usual, this can be fixed by replacing $B$ by the ``overconvergent unit ball''. Covering $\A$ by overconvergent unit balls of increasing radii, we get the desired result.

To simplify notation, let us fix a trivialisation $\Z_p(1)\cong \Z_p$
We start by constructing the explicit perfectoid Galois cover $\wt B\to B$. For this we use the pro-\'etale perfectoid $\Z_p$-torsor
\[\wt \G_m=\textstyle\varprojlim_{[p]} \G_m\to \G_m.\]

\begin{Lemma}\label{l:perfectoid-disc-Galois-cover-of-disc} Embed $B\hookrightarrow \G_m$ via $T\mapsto 1+pT$. Then the pullback $\wt B:=B\times_{\G_m}\wt \G_m$ of $B$ along $\wt \G_m\to\G_m$ is isomorphic over $K$ to the perfectoid unit disc $\Spa K\langle X^{1/p^\infty}\rangle$.
\end{Lemma}
Here $T$ and $X$ are different formal variables, and the map is not given by sending $T\mapsto X$.
\begin{proof}
	We have
	$\wt B=\Spa(R,R^\circ)$ where
	$R:=K\langle Y^{\pm 1/p^\infty}\rangle\langle \tfrac{Y-1}{p}\rangle$, thus $\wt B$
	is affinoid perfectoid. Let $p^\flat\in K^\flat$ be such that $|p^\flat|=|p|$. Write $Y'$ for the parameter of $\G_{m,K^\flat}$. Then 
	\[|Y(x)-1|\leq |p| \Leftrightarrow |Y'(x^\flat)-1|\leq |p^{\flat}| \quad \text{ for any }x\in \wt \G_{m}. \]
	This shows $\wt B^{\flat}=\Spa(R^{\flat},R^{\flat\circ})$ where $R^\flat=K^{\flat}\langle Y'^{\pm 1/p^\infty}\rangle\langle \tfrac{Y'-1}{p^{\flat}}\rangle$.
	But $R^\flat$  is isomorphic to $K^{\flat}\langle X^{ 1/p^\infty}\rangle$ via the map that sends $X^{1/p^n}\mapsto (\tfrac{Y'-1}{p^{\flat}})^{1/p^n}$.
\end{proof}

\begin{Lemma}\label{l:H^1_v(wt B x B, O)}
\label{enum:H^1_v(wt B x B, O)-H1v} For any affinoid perfectoid space $X$ over $K$, there is a natural isomorphism
		\[ H^1_v(X\times B,\O)=\O(X)\hotimes_K H^0(B,\wt\Omega^1).\]
		In particular, for any profinite set $S$, we have $H^1_v(S\times X\times B,\O)=\Map_{\cts}(S,H^1_v(X\times B,\O))$.
\end{Lemma}
\begin{proof}
	By Corollary~\ref{c:CL-acyclic}, the $\Z_p$-Galois cover $X\times \wt B\to X\times B$ induces isomorphisms
	\[H^1_v(X\times B,\O)=H^1_{\cts}(\Z_p,\O(X\times \wt B))=\O(X)\hotimes H^1_{\cts}(\Z_p,\O(\wt B)),\]
	where the last isomorphism is from \cite[Lemma~5.5]{Scholze_p-adicHodgeForRigid}. Then we use that for $X=\Spa(K)$ we already know that $H^1_{\cts}(\Z_p,\O(\wt B))=H^1_v(B,\O)=H^0(B,\wt\Omega^1)$.
\end{proof}

\begin{Lemma}\label{l:Pic_v(wt B x B)-torfree}
	 For any profinite set $S$, the logarithm defines an injection
		\[\log:H^1_v(\wt B\times B\times {S},U)\hookrightarrow H^1_v(\wt B\times B\times {S},\O).\]
	In particular, the specialisation $H^1_v(\wt B\times B\times {S},U)\to \Map(S,H^1_v(\wt B\times B,U))$
		is injective.
\end{Lemma}
\begin{proof}
	Choose a profinite presentation $S=\varprojlim_{i\in I} S_i$.
	We need to show that the map
	\[H^1_{\et}(\wt B\times B\times {S},\mu_{p^\infty})\to H^1_v(\wt B\times B\times {S},U)\]
	vanishes. For this, we use that by Lemma~\ref{l:perfectoid-disc-Galois-cover-of-disc} there is an isomorphism of diamonds
	\[ \wt B\times B\times {S} =\textstyle\varprojlim_{n\in \N,i\in I} B^{(n)}\times B\times {S_i}\]
	where $B^{(n)}$ is $B$ in the variable $X^{1/p^n}$. By \cite[Proposition~14.9]{etale-cohomology-of-diamonds}, we have
	\[ H^1_{\et}(\wt B\times B\times {S},\mu_{p^\infty})= \textstyle\varinjlim_{n,i}H^1_{\et}(B^{(n)}\times B\times S_i,\mu_{p^\infty}).\]
	The result now follows as by Lemma~\ref{l:bOx-on-B^n}, we have $H^1_{\et}(B^2\times S_i,U)\hookrightarrow \Pic_{\et}(B^2\times S_i)=1$.
\end{proof}
\begin{Lemma}\label{l:Picv(wtB x A)}
	Let $Y$ be either of $B$ or $\wt B$. Then
	\[\textstyle\Pic_v( Y\times \A)=H^1_v(Y\times \A,U)=\varprojlim_{s\in \N}H^1_v( Y\times B_s,U).\]
\end{Lemma}
\begin{proof}
	Arguing exactly like in Lemma~\ref{l:H(A^n,U)=H(A^n,Ox)}, for the first equality it suffices to prove that $\bOx(Y\times \A)=K^\times/(1+\m)$ and $H^1_v(Y\times \A,\bOx)=1$. For $B\times \A$,  this can be seen exactly like in Lemma~\ref{l:H(A^n,U)=H(A^n,Ox)}. To deduce the case of $\wt B\times \A$, write $\wt B\sim\varprojlim_n B^{(n)}$ as a tilde-limit where $ B^{(n)}$ is the unit disc in the parameter $X^{1/p^n}$, then by Lemma~\ref{l:R^1nu_*bOx},
	\[\textstyle H^i_v(\wt B\times \A,\bOx)=\varinjlim_{n\in\N} H^i_{v}(B^{(n)}\times \A,\bOx)\quad \text{ for }i=0,1.\]
	The second equality  follows from the \cH-to-sheaf sequence  by the following lemma:
\end{proof}
\begin{Lemma}\label{l:cech-cover-of-XxA}
	Let $Y$ be any affinoid rigid or perfectoid space. Then for the cover $\mathfrak U=(Y\times B_s)_{s\in \N}$ of $Y\times \A$, we have
	 $\cH^i(\mathfrak U,U)=\cH^i(\mathfrak U,\O^\times)=1$ for $i\geq 1$.
\end{Lemma}
\begin{proof}
	The vanishing for $i\geq 2$ follows from $\mathfrak U$ being an increasing cover indexed over $\N$. For $i\geq 1$, it suffices by Lemma~\ref{l:bOx-on-B^n} to see that for $R=\O(Y)$, the following map is surjective:
	\[\textstyle \prod_{s\geq 1}(1+\m R^\circ\langle p^{s}T\rangle)\to \prod_{s\geq 1} (1+\m R^\circ\langle p^{s}T\rangle), \quad (f_s)_{s\in \N}\mapsto (f_sf_{s+1}^{-1})_{s\in \N}.\]
	This can be seen like in \cite[Lemma~6.3.1]{FvdP}: Let  $g=(g_s)_{s\in \N}$ be an element on the right. After rescaling by an element of $(1+\m R^{\circ})_{s\in \N}$, we can assume that
	$g_s = 1+p^{s}T(\dots)$. Then for any $r\in \N$, the product
	$f_r:=\prod_{s\geq r} g_s$
	converges and defines a preimage $(f_r)_{r\in \N}$.
\end{proof}

\subsection{The overconvergent Picard group of the cylinder $B\times \A$}
In the rigid setting, the de Rham complex of the closed unit disc $B$ is not exact on global sections, the issue being convergence of primitive functions at the boundary. It is well-known that this can be resolved by considering overconvergent functions. For $B\times \A$, this means:
\begin{Lemma}\label{l:deRham-cpx-open-unit-disc}
	Recall that $B_s$ is the disc of radius $s\in |K|$. The de Rham complex of $B_s\times \A$
		\[0\to K\to \O(B_s\times \A)\xrightarrow{d} \Omega^{1}(B_s\times \A)\xrightarrow{d} \Omega^{2}(B_s\times \A)\to \dots\]
		becomes exact after applying $\varinjlim_{s>1}$.
\end{Lemma}

The key calculation  is now that of the ``overconvergent Picard group'' of $B\times \A$:
\begin{Proposition}\label{p:Pic=closed-differentials-on-oc-unit-disc}
	The Hodge--Tate logarithm from Theorem~\ref{t:Picv-ses}.1 defines an isomorphism
	\[ \HT\log:\textstyle\varinjlim_{s>1}\Pic_v(B_s\times \A)\isomarrow \textstyle\varinjlim_{s>1}H^0(B_s\times \A,\wt\Omega^1)^{d=0}. \]
\end{Proposition}
\begin{proof}
	By Lemma~\ref{l:cech-cover-of-XxA}, we have $\Pic_{\et}(B_s\times \A)=\varprojlim_{r\in\N}\Pic_{\et}(B_s\times B_r)=1$.
	It therefore follows from Theorem~\ref{t:Picv-ses}.1 that the Hodge--Tate logarithm is an injective map
	\[\HT\log:\Pic_v(B_s\times \A)\hookrightarrow H^1_v(B_s\times \A,\O)=H^0(B_s\times \A,\wt\Omega^1).\]
	We already know that the image contains all closed differentials: By Lemma~\ref{l:deRham-cpx-open-unit-disc}, \[\textstyle\varinjlim_{s>1}H^0(B_s\times\A,\Omega^1)^{d=0}=\textstyle\varinjlim_{s>1}d(\O(B_s\times \A)),\]
	which we know is in the image by Corollary~\ref{c:df-in-image}.
	
	To prove the converse, we start by considering the Cartan--Leray sequences for $U$ and $\O$ associated to
	$\wt B\times B\to B\times B$.
 Lemmas~\ref{l:H^1_v(wt B x B, O)} and \ref{l:Pic_v(wt B x B)-torfree} guarantee that the conditions of Proposition~\ref{p:Cartan--Leray}.\ref{i:CL-five-ex} are satisfied, so the logarithm induces a morphism of short exact sequences
	\[\begin{tikzcd}[row sep = 0.5cm]
			0 \arrow[r] & {H^1_{\cts}(\Z_p,U(\wt B\times B))} \arrow[d] \arrow[r] & H^1_v( B\times B,U) \arrow[d] \arrow[r] & H^1_v(\wt B\times B,U)^{\Z_p} \arrow[d,hook']\arrow[r]&0\\
			0 \arrow[r] & {H^1_{\cts}(\Z_p,\O(\wt B\times B))} \arrow[r] & {H^1_v( B\times B,\O)} \arrow[r] & {H^1_v(\wt B\times B,\O)}^{\Z_p}\arrow[r]&0.
		\end{tikzcd}\]
	The right vertical map is injective by Lemma~\ref{l:Pic_v(wt B x B)-torfree}. The bottom row can be identified with
	\begin{equation}
		0\to  H^0(B,\Omega^1)\hotimes_K \O(B)\to H^0(B\times B,\Omega^1)\to \O(\wt B)\hotimes_K H^0(B,\Omega^1)\to 0
	\end{equation}
	by Lemma~\ref{l:H^1_v(wt B x B, O)}. This expresses that any differential decomposes as $g(T_1,T_2)dT_1+f(T_1,T_2)dT_2$ where $T_1$ is the differential on the first factor of $B\times B$ and $T_2$ is that on the second.
	
	We now replace $B$ by $B_s$, then by Lemma~\ref{l:Picv(wtB x A)} we get for $s\to \infty$ a commutative diagram
	\begin{equation}\label{dg:proof-of-Picv(A^n)-2}
		\begin{tikzcd}[column sep = 0.3cm,row sep = 0.7cm]
			0 \arrow[r] & {H^1_{\cts}(\Z_p,\O^\times(\wt B\times \A))} \arrow[d,hook'] \arrow[r] & \Pic_v( B\times \A) \arrow[d,hook'] \arrow[r] & \Pic_v(\wt B\times \A)^{\Z_p} \arrow[d,hook']&\\
			0\arrow[r] &  H^0(B,\Omega^1)\hotimes \O(\A)\arrow[r] & H^0(B\times \A,\Omega^1)\arrow[r] & \O(B)\hotimes H^0(\A,\Omega^1)\arrow[r] & 0,
		\end{tikzcd}
	\end{equation}
	here the top left entry is  as described by the Cartan--Leray sequence of
	$\wt B\times \A\to B\times \A$.
	
	We now have a closer look at the left vertical map: Here $\O^\times(\wt B\times \A)=\O^\times(\wt B)$, 
	for which
	\[ \HT\log:H^1_{\cts}(\Z_p,\O^\times(\wt B))\to\Pic_v(B)\xrightarrow{}H^0(B,\Omega^1)\]
	is an isomorphism by Theorem~\ref{t:Picv-ses}.\ref{enum:MT1-curves}. We conclude that the image of the leftmost vertical map in diagram~\eqref{dg:proof-of-Picv(A^n)-2} consists precisely of the submodule $H^0(B,\Omega^1)$.
	
	Next, we replace the first factor $B$ by  $B_s$ which in the colimit $s\to 1$ results in a diagram
	\[\begin{tikzcd}[column sep = 0.2cm]
			0 \arrow[r] & \varinjlim_{s> 1}{H^0(B_s,\Omega^1)} \arrow[d,hook'] \arrow[r] & \varinjlim_{s> 1}\Pic_v( B_s\times \A) \arrow[d,hook'] \arrow[rd,dashed]\arrow[r]&\coker\arrow[d,hook']\arrow[r]&0\\
			0\arrow[r] &  \varinjlim_{s> 1}H^0(B_s,\Omega^1)\hotimes \O(\A)\arrow[r] & \varinjlim_{s> 1}H^0(B_s\times \A,\Omega^1)\arrow[r] & \varinjlim_{s> 1}\O(B_s)\hotimes H^0(\A,\Omega^1)\arrow[r] & 0
		\end{tikzcd}\]
	with exact rows. 
	We claim that the bottom sequence restricts to an exact sequence
	\[\textstyle 0\to \varinjlim_{s>1}H^0(B_s,\Omega^1)\to \varinjlim_{s>1}H^0(B_s\times \A,\Omega^1)^{d=0}\to \varinjlim_{s> 1}\O(B_s)\hotimes H^0(\A,\Omega^1)\to 0.\]
	Left-exactness is clear. For any differential $\omega_2=f(T_1,T_2)dT_2$ in $\O(B_s)\hotimes H^0(\A,\Omega^1)$, we can find $h(T_1,T_2)\in \O(B_s\times \A)$ such that $\partial h/\partial T_2=-\partial f/\partial T_1$. For any such function, $\omega:=hdT_1+fdT_2$ is a closed differential on $B_s\times \A$ which the last map sends to $\omega_2$. 
	 
	By Lemma~\ref{l:deRham-cpx-open-unit-disc}, we can even find $g\in \varinjlim_{s>1}\O(B_s\times\A)$ such that $\omega=dg$. Thus the dashed map is surjective, hence the right vertical map is an isomorphism. 
	Comparing  the image of the top row in the bottom row with the last exact sequence, using that the image of $\Pic_v(B_s\times \A)$ contains $H^0(B_s\times \A,\Omega^1)^{d=0}$, this shows that inside $H^0(B_s\times \A,\Omega^1)$, we have
	\[\textstyle\varinjlim_{s>1} \Pic_v(B_s\times \A)= \varinjlim_{s>1} H^0(B_s\times \A,\Omega^1)^{d=0}.\qedhere\]
\end{proof}
\subsection{The Picard group of $\A^n$}

\begin{proof}[Proof of Theorem~\ref{t:Pic-v-A^n}]
	Since $\Pic_{\an}(\A^n)=1$, we have by Theorem~\ref{t:Picv-ses} an injective map
	\[ \HT\log:\Pic_v(\A^n)\hookrightarrow H^0(\A^n,\wt\Omega^1).\]
	It is clear from Corollary~\ref{c:df-in-image} and exactness of the de Rham complex of $\A^n$ that 
	\[ H^0(\A^n,\wt\Omega^1)^{d=0}=d(\O(\A^n))\subseteq \im(\HT\log).\]
	To prove the other containment, we first consider the case of $n=2$: In this case, the restriction from $\A\times \A$ to $B_s\times \A$ for any $s>1$ defines a commutative diagram
	\[
	\begin{tikzcd}
		\Pic_v(\A\times \A) \arrow[d] \arrow[r,hook]& {H^0(\A\times \A,\wt\Omega^1)} \arrow[d,hook]  \\
		\varinjlim_{s>1}\Pic_v(B_s\times \A)                \arrow[r,hook] & \varinjlim_{s>1}{H^0(B_s\times \A,\wt\Omega^1)}.
	\end{tikzcd}\]
	By Proposition~\ref{p:Pic=closed-differentials-on-oc-unit-disc}, the image of the bottom map lands in the closed differentials.
	As the map on the right is injective, this also holds for the top map. This proves the case of $n=2$.
	
	The general case follows from the one of $n=2$: Let $f\in H^0(\A^n,\widetilde\Omega^1)$
	be in the image of $\HT\log$
	and suppose that $df\neq 0$. We can write this  as
	\[ df = \textstyle\sum_{i<j}g_{ij}dX_i\wedge dX_j.\]
	Then $df\neq 0$ if and only if there is some $0\neq g_{ij}\in \O(\A^n)$, and after reordering we can assume $0\neq g_{12}:\A^n(K)\to K$ is non-trivial. 
	We may thus find $z\in \A^{n-2}$ such that under
	\[ \varphi:\A^{2}\xrightarrow{(\id,z)} \A^{2}\times \A^{n-2},\]
	 $f$ pulls back to $\varphi^{\ast}f$ which still satisfies $d(\varphi^{\ast}f)=\varphi^{\ast}df\neq 0$. But the commutative diagram
	\[\begin{tikzcd}
		\Pic_v(\A^n) \arrow[r,"\varphi^\ast"] \arrow[d] & \Pic_v(\A^2) \arrow[d] \\
		{H^0(\A^n,\wt\Omega^1)} \arrow[r,"\varphi^\ast"] & {H^0(\A^2,\wt\Omega^1)}
	\end{tikzcd}\]
	shows that $\varphi^{\ast}f$ is in the image of $\Pic_v(\A^2)$, which implies $d(\varphi^{\ast}f)=0$, a contradiction.
	
	The case of $\Pic_v(\G_m^k\times \A^n)$ is analogous: Here we first note that $H^0(\G_m^k\times \A^n,\wt\Omega^1)^{d=0}$ is generated as a group by $d(\O(\G_m^k\times \A^n))$ plus the differentials $a\cdot dY_i/Y_i$ for $a\in K$ and for each of the $\G_m$-factors. The latter are in the image of $\Pic_v(\G_m^k\times \A^n)$ as we see via pullback along the projection $\G_m^k\times \A^n\to \G_m$ since $\Pic_v(\G_m)=H^0(\G_m,\wt\Omega^1)$ by Theorem~\ref{t:Picv-ses}.\ref{enum:MT1-curves}.
	The rest of the proof goes through analogously by considering any embedding $B_s\hookrightarrow \G_m$.
\end{proof}

\bibliographystyle{alphabbrv}
\bibliography{universal.bib}

\newcommand{\etalchar}[1]{$^{#1}$}
\begin{thebibliography}{BGH{\etalchar{+}}18}

\bibitem[AGT16]{AGT-p-adic-Simpson}
A.~Abbes, M.~Gros,  T.~Tsuji.
\newblock {\em The {$p$}-adic {S}impson correspondence}, volume 193 of {\em
  Annals of Mathematics Studies}.
\newblock Princeton University Press, 2016.

\bibitem[Ber93]{Berkovich_EtaleCohom}
V.~G. Berkovich.
\newblock \'{E}tale cohomology for non-{A}rchimedean analytic spaces.
\newblock {\em Inst. Hautes \'{E}tudes Sci. Publ. Math.}, (78):5--161 (1994),
  1993.

\bibitem[BGH{\etalchar{+}}18]{perfectoid-covers-Arizona}
C.~Blakestad, D.~Gvirtz, B.~Heuer, D.~Shchedrina, K.~Shimizu, P.~Wear,  Z.~Yao.
\newblock {P}erfectoid covers of abelian varieties.
\newblock {\em Preprint, ar{X}iv:1804.04455}, 2018.

\bibitem[BGR84]{BGR}
S.~Bosch, U.~G\"{u}ntzer,  R.~Remmert.
\newblock {\em {N}on-{A}rchimedean analysis}, volume 261 of {\em Grundlehren
  der Mathematischen Wissenschaften}.
\newblock Springer-Verlag, Berlin, 1984.

\bibitem[BHW18]{ChrisChris-HilbertCHJ}
C.~Birkbeck, B.~Heuer,  C.~Williams.
\newblock {O}verconvergent {H}ilbert modular forms via perfectoid {m}odular
  {v}arieties.
\newblock {\em Preprint, arXiv:1902.03985}, 2018.

\bibitem[BMS18]{BMS}
B.~Bhatt, M.~Morrow,  P.~Scholze.
\newblock {I}ntegral {$p$}-adic {H}odge theory.
\newblock {\em Publ. Math. Inst. Hautes \'{E}tudes Sci.}, 128:219--397, 2018.

\bibitem[Bos19]{Bosco_p-adicCohomology}
G.~Bosco.
\newblock $p$-adic cohomology of {D}rinfeld spaces.
\newblock Master's thesis, Universit\'e Paris-Sud, 2019.

\bibitem[BS15]{bhatt-scholze-proetale}
B.~Bhatt,  P.~Scholze.
\newblock {T}he pro-\'{e}tale topology for schemes.
\newblock {\em Ast\'{e}risque}, (369):99--201, 2015.

\bibitem[CHJ17]{CHJ}
P.~Chojecki, D.~Hansen,  C.~Johansson.
\newblock {O}verconvergent modular forms and perfectoid {S}himura curves.
\newblock {\em Doc. Math.}, 22:191--262, 2017.

\bibitem[CN20]{ColmezNiziol_CohomologyAffine}
P.~Colmez,  W.~Nizio{\l}.
\newblock On the cohomology of the affine space.
\newblock In {\em p-adic Hodge Theory}, pages 71--80. Springer, 2020.

\bibitem[Col02]{Colmez-espacesdeBanach}
P.~Colmez.
\newblock {E}spaces de {B}anach de dimension finie.
\newblock {\em J. Inst. Math. Jussieu}, 1(3):331--439, 2002.

\bibitem[dJ95]{dJ_etalefundamentalgroups}
A.~J. de~Jong.
\newblock \'{E}tale fundamental groups of non-{A}rchimedean analytic spaces.
\newblock volume~97, pages 89--118. 1995.
\newblock Special issue in honour of Frans Oort.

\bibitem[dJ{\etalchar{+}}20]{StacksProject}
A.~J. de~Jong, et~al.
\newblock {T}he stacks project.
\newblock 2020.

\bibitem[dJvdP96]{deJongvdPut}
A.~J. de~Jong,  M.~van~der Put.
\newblock \'{E}tale cohomology of rigid analytic spaces.
\newblock {\em Doc. Math.}, 1:No. 01, 1--56, 1996.

\bibitem[DW05]{DeningerWerner_vb_p-adic_curves}
C.~Deninger,  A.~Werner.
\newblock {V}ector bundles on {$p$}-adic curves and parallel transport.
\newblock {\em Ann. Sci. \'{E}cole Norm. Sup. (4)}, 38(4):553--597, 2005.

\bibitem[DW20]{DeningerWerner_Simpson}
C.~Deninger,  A.~Werner.
\newblock Parallel transport for vector bundles on {$p$}-adic varieties.
\newblock {\em J. Algebraic Geom.}, 29(1):1--52, 2020.

\bibitem[Fal05]{Faltings_SimpsonI}
G.~Faltings.
\newblock {A} {$p$}-adic {S}impson correspondence.
\newblock {\em Adv. Math.}, 198(2):847--862, 2005.

\bibitem[FvdP04]{FvdP}
J.~Fresnel,  M.~van~der Put.
\newblock {\em {R}igid analytic geometry and its applications}, volume 218 of
  {\em Progress in Mathematics}.
\newblock Birkh\"{a}user Boston, Inc., Boston, MA, 2004.

\bibitem[Gru68]{Gruson_FibresVect}
L.~Gruson.
\newblock {F}ibr\'{e}s vectoriels sur un polydisque ultram\'{e}trique.
\newblock {\em Ann. Sci. \'{E}cole Norm. Sup. (4)}, 1:45--89, 1968.

\bibitem[Heua]{heuer-Picard-good-reduction}
B.~Heuer.
\newblock Line bundles on perfectoid covers: good reduction.
\newblock {\em In preparation}.

\bibitem[Heub]{heuer-isoclasses}
B.~Heuer.
\newblock {P}ro-\'etale uniformisation of abelian varieties.
\newblock {\em In preparation}.

\bibitem[Heu21]{heuer-diamantine-Picard}
B.~Heuer.
\newblock {D}iamantine {P}icard functors of rigid spaces.
\newblock {\em arXiv:2103.16557}, 2021.

\bibitem[HL20]{HansenLi_HodgeSymmetry}
D.~Hansen,  S.~Li.
\newblock Line bundles on rigid varieties and {H}odge symmetry.
\newblock {\em Math. Z.}, 296(3-4):1777--1786, 2020.

\bibitem[Hub96]{huber2013etale}
R.~Huber.
\newblock {\em \'{E}tale cohomology of rigid analytic varieties and adic
  spaces}.
\newblock Aspects of Mathematics, E30. Friedr. Vieweg \& Sohn, Braunschweig,
  1996.

\bibitem[Kat73]{p-adicMSMF}
N.~M. Katz.
\newblock {$p$}-adic properties of modular schemes and modular forms.
\newblock In {\em Modular functions of one variable, {III}}, pages 69--190.
  Lecture Notes in Math., Vol. 350. Springer, Berlin, 1973.

\bibitem[KL16]{KedlayaLiu-II}
K.~S. Kedlaya,  R.~Liu.
\newblock {{R}elative $p$-adic {H}odge theory, {I}{I}: {I}mperfect period
  rings}.
\newblock {\em Preprint, arXiv:1602.06899}, 2016.

\bibitem[LB18]{LeBras-Espaces}
A.-C. Le~Bras.
\newblock {E}spaces de {B}anach-{C}olmez et faisceaux coh\'{e}rents sur la
  courbe de {F}argues-{F}ontaine.
\newblock {\em Duke Math. J.}, 167(18):3455--3532, 2018.

\bibitem[L{\"{u}}t77]{Lutkebohmert_Vektorraumbundel}
W.~L{\"{u}}tkebohmert.
\newblock {V}ektorraumb\"{u}ndel \"{u}ber nichtarchimedischen holomorphen
  {R}\"{a}umen.
\newblock {\em Math. Z.}, 152(2):127--143, 1977.

\bibitem[LvdP95]{LiuvdPut_separatedcurves}
Q.~Liu,  M.~van~der Put.
\newblock {O}n one-dimensional separated rigid spaces.
\newblock {\em Indag. Math. (N.S.)}, 6(4):439--451, 1995.

\bibitem[LZ17]{LiuZhu_RiemannHilbert}
R.~Liu,  X.~Zhu.
\newblock Rigidity and a {R}iemann-{H}ilbert correspondence for {$p$}-adic
  local systems.
\newblock {\em Invent. Math.}, 207(1):291--343, 2017.

\bibitem[MW20]{MannWerner_LocSys_p-adVB}
L.~Mann,  A.~Werner.
\newblock {L}ocal systems on diamonds and $p$-adic vector bundles, 2020.

\bibitem[NSW08]{NeuSchWin}
J.~Neukirch, A.~Schmidt,  K.~Wingberg.
\newblock {\em {C}ohomology of number fields}, volume 323 of {\em Grundlehren
  der Mathematischen Wissenschaften}.
\newblock Springer-Verlag, Berlin, second edition, 2008.

\bibitem[Rob67]{Robertson_toptorsion}
L.~C. Robertson.
\newblock {C}onnectivity, divisibility, and torsion.
\newblock {\em Trans. Amer. Math. Soc.}, 128:482--505, 1967.

\bibitem[Sch13a]{Scholze_p-adicHodgeForRigid}
P.~Scholze.
\newblock {$p$}-adic {H}odge theory for rigid-analytic varieties.
\newblock {\em Forum Math. Pi}, 1:e1, 77, 2013.

\bibitem[Sch13b]{Scholze2012Survey}
P.~Scholze.
\newblock {P}erfectoid spaces: a survey.
\newblock In {\em Current developments in mathematics 2012}, pages 193--227.
  Int. Press, Somerville, MA, 2013.

\bibitem[Sch18]{etale-cohomology-of-diamonds}
P.~Scholze.
\newblock {\'E}tale cohomology of diamonds.
\newblock {\em Preprint, arXiv:1709.07343}, 2018.

\bibitem[Sch19]{Scholze-Condensed}
P.~Scholze.
\newblock {L}ectures on condensed mathematics.
\newblock {\em University of Bonn lecture notes}, 2019.
\newblock www.math.uni-bonn.de/people/scholze/Condensed.pdf.

\bibitem[Sim92]{SimpsonCorrespondence}
C.~T. Simpson.
\newblock Higgs bundles and local systems.
\newblock {\em Inst. Hautes \'{E}tudes Sci. Publ. Math.}, (75):5--95, 1992.

\bibitem[SW20]{ScholzeBerkeleyLectureNotes}
P.~Scholze,  J.~Weinstein.
\newblock {\em $p$-adic geometry, {{U}C} {B}erkeley course notes}.
\newblock Annals of Mathematics Studies. Princeton University Press, Princeton,
  NJ, 2020.

\bibitem[Wü20]{wuerthen_vb_on_rigid_var}
M.~Würthen.
\newblock {V}ector bundles with numerically flat reduction on rigid analytic
  varieties and $p$-adic local systems.
\newblock {\em Preprint, arXiv:1910.03727}, 2020.

\end{thebibliography}
\end{document}